\documentclass[journal]{IEEEtran}

\ifCLASSINFOpdf
\usepackage[latin1]{inputenc}
\usepackage{amsmath,psfrag}   
\usepackage{graphicx}         
\usepackage{subfigure}			
\usepackage{upgreek}			  
\usepackage{mathtools}		  
\usepackage{url}					
\usepackage{enumitem}		
\usepackage{accents}			
\usepackage{bbm}
\usepackage{xcolor}
\usepackage{lipsum}

\else
\fi

\usepackage{amsmath}
\usepackage{amsthm}
\newtheoremstyle{break}
  {9pt}
  {9pt}
  {\itshape}
  {}
  {\bfseries}
  {.}
  { }
  {}
\theoremstyle{break}

\newtheorem{thm}{Theorem}[section]

\newtheorem{prop}[thm]{Proposition}

\newtheorem{defn}[thm]{Definition}

\newtheorem{alg}[thm]{Algorithm}

\DeclareMathOperator*{\argmin}{arg\,\min\;}

\usepackage{mathtools}		  
\usepackage{amssymb}		  
\usepackage{url}
\usepackage{bbm}
\usepackage{xcolor}
\usepackage{lipsum}
\usepackage{booktabs}

\begin{document}
%
\title{A New Sparse and Robust Adaptive Lasso Estimator for the Independent Contamination Model}
%
%
%


\author{\IEEEauthorblockN{Jasin Machkour, Michael Muma,~\IEEEmembership{Member,~IEEE,} Bastian Alt, Abdelhak M. Zoubir,~\IEEEmembership{Fellow,~IEEE}}
\thanks{J. Machkour, M. Muma and A.M. Zoubir are with the Signal Processing Group Technische Universit\"at Darmstadt. Merckstr. 25, 64283 Darmstadt, Germany. Email: \{machkour,muma,zoubir\}@spg.tu-darmstadt.de. B. Alt is with the Bioinspired Communication Systems, Rundeturmstr. 12, 64283 Darmstadt, Germany. Email:  bastian.alt@bcs.tu-darmstadt.de}
}


%
%

\markboth{IEEE Transactions on Signal Processing,~Vol.~XX, No.~XX, XXX~XXXX}%
{}
%



\maketitle

\begin{abstract}
Many problems in signal processing require finding sparse solutions to under-determined, or ill-conditioned, linear systems of equations. When dealing with real-world data, the presence of outliers and impulsive noise must also be accounted for. In past decades, the vast majority of robust linear regression estimators has focused on robustness against rowwise contamination. Even so called `high breakdown' estimators rely on the assumption that a majority of rows of the regression matrix is not affected by outliers. Only very recently, the first cellwise robust regression estimation methods have been developed. In this paper, we define robust oracle properties, which an estimator must have in order to perform robust model selection for under-determined, or ill-conditioned linear regression models that are contaminated by cellwise outliers in the regression matrix. We propose and analyze a robustly weighted and adaptive Lasso type regularization term which takes into account cellwise outliers for model selection. The proposed regularization term is integrated into the objective function of the MM-estimator, which yields the proposed MM-Robust Weighted Adaptive Lasso (MM-RWAL), for which we prove that at least the weak robust oracle properties hold. A performance comparison to existing robust Lasso estimators is provided using Monte Carlo experiments. Further, the MM-RWAL is applied to determine the temporal releases of the European Tracer Experiment (ETEX) at the source location. This ill-conditioned linear inverse problem contains cellwise and rowwise outliers and is sparse both in the regression matrix and the parameter vector. The proposed RWAL penalty is not limited to the MM-estimator but can easily be integrated into the objective function of other robust estimators.
\end{abstract}


%
\IEEEpeerreviewmaketitle

\section{Introduction}
\label{sec:introduction}
Many of today's signal processing problems can be formulated as a linear regression 
\begin{equation}
\mathbf{y}=\mathbf{X}\boldsymbol{\beta}+\mathbf{u},
\label{eqn:linear_regression_matrix}
\end{equation}
where we assume that the regressor matrix $\mathbf{X}\in\mathbb{R}^{n\times p}$, the errors $\mathbf{u}\in \mathbb{R}^{n\times 1}$ and observations $\mathbf{y}\in \mathbb{R}^{n\times 1}$ are independent and identically distributed (iid) random variables, $\boldsymbol{\beta}\in\mathbb{R}^{p\times 1}$ are the unknown parameters of interest, and $\mathbf{X}$ and $\mathbf{u}$ are mutually independent. 

The presence of outliers and impulsive noise has been reported in applications as diverse as wireless communication, ultrasonic systems, computer vision, electric power systems, automated detections of defects, biomedical signal analysis, genomics and the estimation of the temporal releases of a pollutant to the atmosphere. See  \cite{zoubir2012,elhamifar2013sparse,liu2013robust,pascal2014,pascal_shrinkage,martinez2014robust,marta_icassp15,so_2016,ollila2016} and references therein. Violations of the Gaussian assumption cause a drastic performance drop for the commonly used least-squares estimator (LSE) \cite{hampel2005, maronna2006, huber2009}
\begin{equation}
\hat{\boldsymbol{\beta}}_{\mathrm{LSE}}=\underset{\boldsymbol{\beta}}{\argmin}\, \|\textbf{y}-\mathbf{x}\boldsymbol{\beta}\|_{2}^{2}.
\label{eq:lse}
\end{equation}

For decades, the vast majority of robust linear regression estimators has focused on robustness against 'rowwise' contamination. Under this so-called Tukey-Huber contamination model (THCM) \cite{maronna2006}, a small fraction of rows of $\mathbf{X}$ may be contaminated. Even 'high-breakdown' regression estimators, such as the LTS-, S-, MM-, and $\tau$-estimators \cite{zoubir2012,maronna2006} rely on the THCM. In \cite{rousseeuw2016detecting}, Rousseuw and Van den Bossche state that {\it recently  researchers  have  come  to  realize  that  the  outlying  rows  paradigm  is no longer sufficient for modern high-dimensional data sets.  It often happens that most data cells (entries) in a row are regular and just a few of them are anomalous.} 

The case that independent cells of $\mathbf{X}$ are outliers is referred to as the independent contamination model (ICM) \cite{alqallaf2009propagation,ollerer2016shooting,leung2016robust}. Only very recently, the first `cellwise robust' regression estimation methods have been developed \cite{ollerer2016shooting,leung2016robust}. The extension of existing estimators to other contamination models, such as the ICM, and even the development of completely new robust estimators is necessary to solve many real-world problems. For example, the estimation of the spatio-temporal emissions of a pollutant, given noisy observations, can be formulated as a linear inverse problem with the help of an atmospheric dispersion model \cite{martinez2014robust}. The data of the European Tracer Experiment (ETEX) which was conducted in Monterfil, Brittany in 1994, where Perfluorocarbon (PFC) tracers were released into the atmosphere, for instance, contains both cellwise and rowwise outliers. 

Additionally to the robustness considerations, atmospheric inverse problems, like many other problems in signal processing, require finding sparse solutions to under-determined, or ill-conditioned, linear systems of equations. For example, handling large datasets in terms of model interpretation, including the case where the number of explanatory variables~$p$ is larger than the sample size~$n$, requires penalized estimators, such as the classical least absolute shrinkage and selection operator (Lasso) \cite{tibshirani1996regression}
\begin{equation}
\hat{\boldsymbol{\beta}}_{\mathrm{Lasso}}=\underset{\boldsymbol{\beta}}{\argmin}\, \|\textbf{y}-\mathbf{x}\boldsymbol{\beta}\|_{2}^{2}+\lambda\|\boldsymbol{\beta}\|_{1}.
\label{eq: lasso lagrange}
\end{equation}
with $\lambda\in\mathbb{R}^+$. 

Many other regularizations have been proposed \cite{friedman2010regularization,huang2012selective,NIPS2013_4904,zhang2012general,lian2016nonconvex}. In this paper the focus lies on Lasso estimation, to select a robust and interpretable model in high dimensional settings. 
Zou \cite{zou2006adaptive} showed that the Lasso variable selection can be inconsistent, so that the oracle properties do not hold and proposed the adaptive Lasso
\begin{equation}
\hat{\boldsymbol{\beta}}_{\mathrm{Lasso}}^{\mathrm{ad}}=\underset{\boldsymbol{\beta}}{\argmin}\, \|\textbf{y}-\mathbf{x}\boldsymbol{\beta}\|_{2}^{2}+\lambda\sum\limits_{j=1}^{p}\hat{w}_{j}|\beta_{j}|,
\label{eq: adaptive lasso}
\end{equation}
where $\hat{w}_{j}=1/|\hat{\beta}_j|^{\gamma}$ ($\gamma>0$) are non-negative weights depending on $\hat{\boldsymbol{\beta}}$, which is a $\sqrt{n}$-consistent estimator of $\boldsymbol{\beta}$. 


Just like the LSE, the Lasso and the adaptive Lasso rely on the Gaussian noise assumption and are sensitive to outliers. In recent years, some robust and regularized approaches have been proposed that replace the penalized square objective function by a penalized bounded objective function \cite{marta_icassp15,ollila2016,Maronna2011sridge,Alfons2013sparseLTS,smucler2015robust}. These methods, however, again, rely on the THCM, and to date, no penalized robust regression method exists that can handle cellwise and rowwise outliers.

\noindent {\it Original Contributions:} 
First, we give a weak and strong definition of what we call the 'robust oracle properties'. These are properties that estimators aiming at performing {\it robust} variable selection need to have. Next, we propose and analyze a robustly weighted and adaptive Lasso-type regularization term, which takes into account cellwise outliers for model selection. The proposed regularization term is integrated into the objective function of the MM-estimator, which yields the proposed \textit{MM-Robust Weighted Adaptive Lasso (MM-RWAL)}, for which we prove that at least the weak robust oracle properties hold. We would like to highlight, that the proposed Robust Weighted Adaptive Lasso penalty can easily be integrated into the objective function of other robust estimators. A performance comparison to existing robust Lasso estimators is provided using Monte Carlo experiments. Further, a challenging real-data application of estimating the sparse non-negative spatio-temporal emissions of a pollutant is considered, given noisy observations $\textbf{y}$ and an imprecisely estimated ill-conditioned and sparse dispersion model $\mathbf{X}$. This example contains both cellwise and rowwise outliers.

\noindent {\it Notation:}
Scalars are denoted by lowercase letters, e.g., $x$, column vectors by bold-faced lowercase letters, e.g. $\mathbf{x}$, matrices by bold-faced uppercase letters, e.g. $\mathbf{X}$, sets are denoted by calligraphic letters, e.g. $\mathcal{X}$ with associated cardinality $|\mathcal{X}|$. The $j$th column of a matrix $\mathbf{X}$ is denoted by $\mathbf{x}_j$ while $(\mathbf{x})_{i:j}$ denotes the vector that contains the entries $i$ to $j$ of vector~$\mathbf{x}$. The $i$th element of vector $\mathbf{x}$ is denoted by $x_i$, $\textbf{I}_{p}$ is the $p$-dimensional identity-matrix, $\mathbf{0}_p$ is the $p$-dimensional all-zeros vector and $\mathrm{diag}(\mathbf{x})$ forms a matrix that contains the entries of $\mathbf{x}$ as its diagonal. $\hat{\boldsymbol{\beta}}$ refers to the estimator (or estimate) of the parameter vector $\boldsymbol{\beta}$, $(\cdot)^\top$ is the transpose operator. The derivative of a function $f$ with respect to its argument is abbreviated by $f^{\prime}$. $P(X)$ is the probability of event $X$. $\mathrm{Bin}(1,\epsilon)$ denotes the binomial distribution with one trial and a success probability of $\epsilon$. Convergence to the normal distribution with mean vector $\boldsymbol{\mu}$ and covariance matrix $\boldsymbol{\Sigma}$ is denoted by $\overset{d\,\,}{\rightarrow}\mathcal{N}(\textbf{0},\boldsymbol{\Sigma})$.

\noindent {\it Organization:}
Section \ref{sec:THCM-ICM} discusses the Tukey-Huber and Independent Contamination models and motivates the use of cellwise robust methods. Section \ref{sec:prop-methods} defines the robust oracle properties, introduces the proposed estimator and provides algorithms to compute the estimates. Section \ref{sec:simulations} provides numerical experiments, while Section \ref{sec:real-data} contains a challenging real-data application of source estimation for an atmospheric inverse problem. Finally, Section \ref{sec:conclusion} concludes the paper with a brief outlook on future work.

\section{Tukey-Huber and Independent Contamination Model}
\label{sec:THCM-ICM}
\subsection{Tukey-Huber Contamination Model (THCM)}
The THCM is based on the assumption that a majority~${(1-\epsilon)}$ of the data points is not contaminated, while a minority $\epsilon$ is contaminated. 
The univariate THCM is formulated as
\begin{equation}
x=(1-b)y + bz,
\label{eq: THCM}
\end{equation}
where $b,\,y,\,z$ are mutually independent scalar random variables. Let $F$ be the true distribution of the data and let $G$ be an unspecified contaminating distribution, from which the outliers are generated. Then, with ${y\sim F}$, ${z\sim G}$ and ${b\sim \mathrm{Bin}(1,\epsilon)}$, the distribution of the observed variable $x$ becomes
\begin{equation}
H = (1-\epsilon)F + \epsilon G.
\end{equation}
The multivariate THCM is defined by
\begin{equation}
\mathbf{x}=(\textbf{I}_{p}-b\textbf{I}_{p})\textbf{y} + b\textbf{I}_{p}\textbf{z},
\label{eq: THCM multivariate}
\end{equation}
where $\textbf{y} ,\, \mathbf{x} ,\, \textbf{z} $ are $p$-dimensional random variables.

A highly valuable property of the THCM is that the percentage of contaminated rows in the data-matrix stays unchanged under affine transformations, that is, if the random vector $\mathbf{x}$ follows the THCM, then the affine transformed vector
\begin{equation}
\tilde{\mathbf{x}}=\textbf{Ax} + \textbf{b}
\end{equation}
also follows the THCM. Thus, estimators designed for the THCM can be affine equivariant. Additionally, many important robustness concepts such as the influence function and the breakdown point are based on the THCM \cite{hampel2005, maronna2006, huber2009, zoubir2012}. However, the disadvantages of this contamination model, which occur especially in high dimensional settings, are alerting. First of all, the assumption that a major fraction $(1-\epsilon)$ of the data points is outlier free, very unlikely holds in higher dimensions. A mathematical legitimation of this criticism will be given in Eq.~\eqref{eq:thcm-icm} and the associated Fig.~\ref{fig: ICM THCM}. Secondly, as illustrated in Fig.~\ref{fig: THCM illustration}a we see that a few cellwise outliers in the THCM lead to flagging the whole corresponding row of $\mathbf{X}$ as outliers. Although much information is lost, such contamination may be handled by high breakdown point estimators in low dimensions. However, Fig.~\ref{fig: THCM illustration}b illustrates that in cases where the number of predictors $p$ exceeds the number of rows $n$, it becomes more and more likely that a few highly contaminated predictors force the THCM-based estimators to flag all data points as outliers, which makes it impossible to draw any inferences from the data.
\begin{figure}[htbp]
\centering
\includegraphics[width=0.48\textwidth]{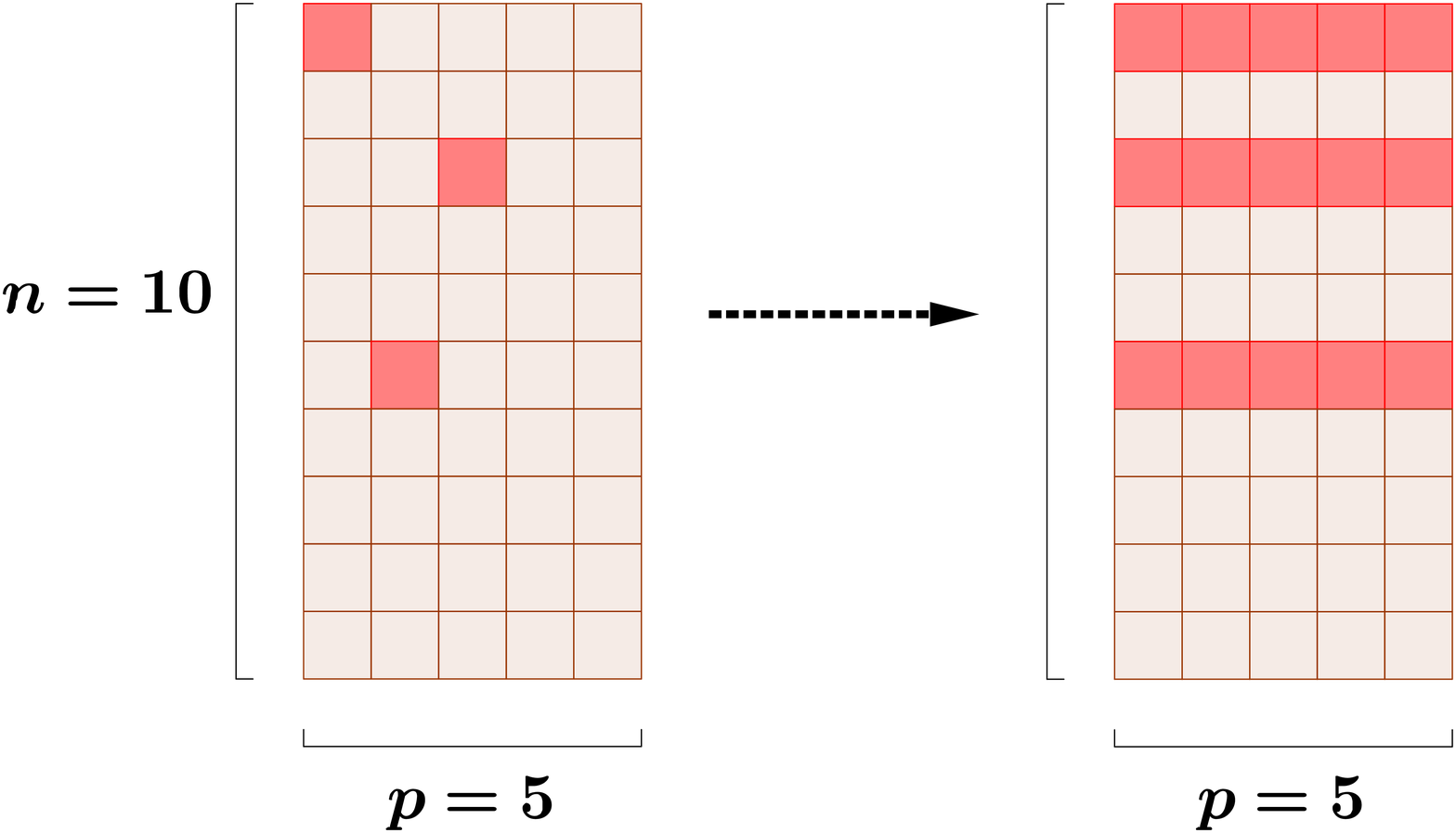}
\\
(a) $p<n.$\\
\includegraphics[width=0.28\textwidth]{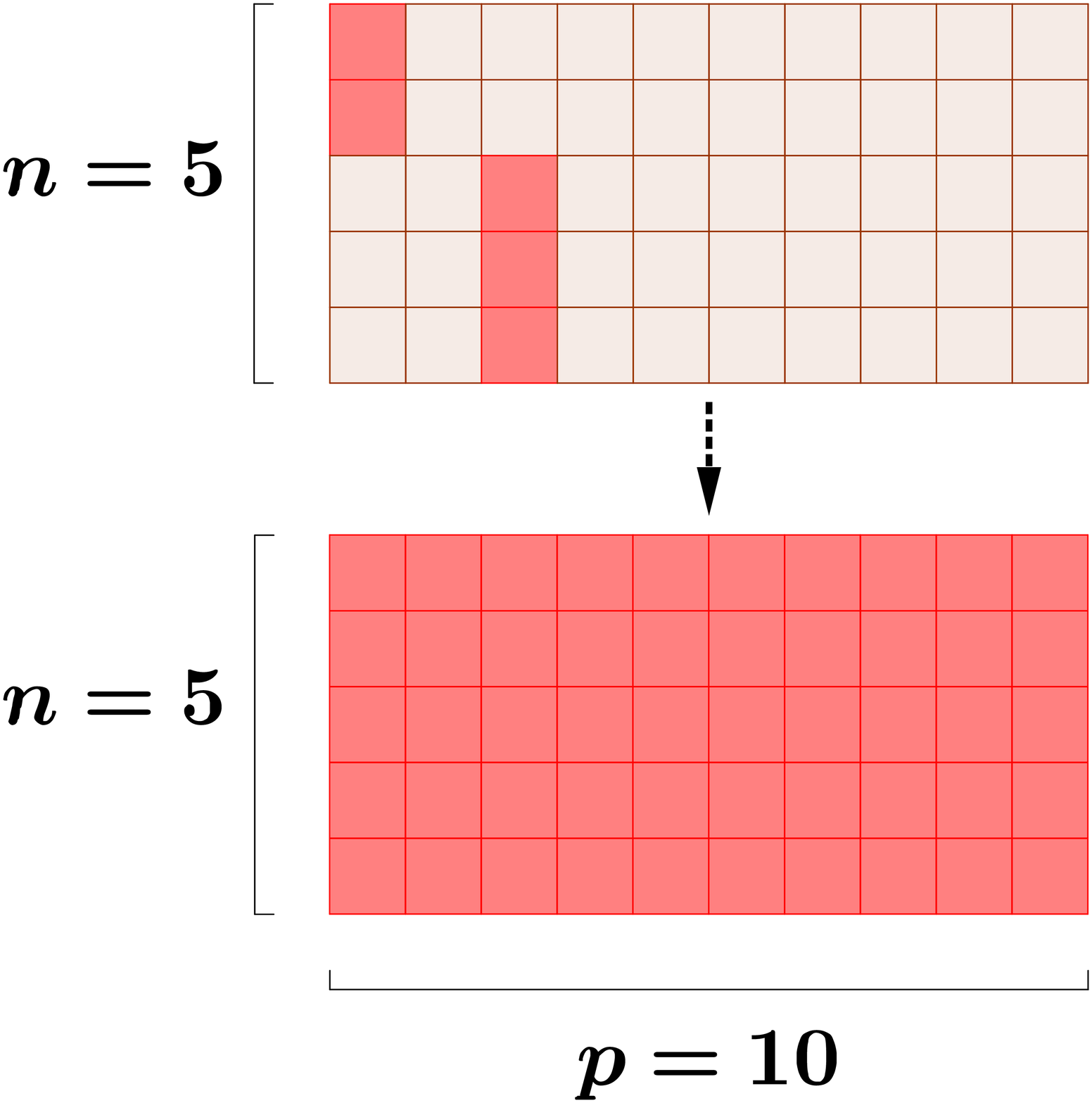}
\\
(b) $p>n.$
\caption{Illustration of how the THCM downweights outliers in the case of (a) a few cellwise outliers for $p<n$ and (b) a few highly contaminated predictors for $p<n$. }
\label{fig: THCM illustration}
\end{figure}

\subsection{Independent Contamination Model (ICM)}
The ICM is defined by 
\begin{equation}
\mathbf{x}=(\textbf{I}_{p}-\textbf{B})\textbf{y} + \textbf{B}\textbf{z},
\label{eq: ICM multivariate}
\end{equation}
where $\textbf{B}=\mathrm{diag}(b_{1}, b_{2},\ldots,b_{p})$ and $b_{1}, b_{2},\ldots,b_{p}$ are independent Bernoulli random variables with success probability $\epsilon_{j},\, j=1,\ldots,p$. Loosely speaking, each cell $x_{ij}$ corresponding to the predictor $j$ in every single row has a probability $\epsilon_{j}$ of being contaminated. Notice that for $P(b_{1}=b_{2}=\ldots=b_{p})=1$, the ICM reduces to the THCM.

The main issue with the ICM is that it is not equivariant under affine transformations. Let $\mathbf{x}$ be a random vector and $\textbf{A}$ an invertible quadratic $p$-dimensional matrix. Then, for an affine transformation of $\mathbf{x}$ 
\begin{align}
\begin{split}
\tilde{\mathbf{x}}=\textbf{Ax} + \textbf{b}&=\textbf{A}(\textbf{I}_{p}-\textbf{B})\textbf{y} + \textbf{ABz} + \textbf{b}\\
&\neq (\textbf{I}_{p}-\textbf{B})\textbf{Ay} + \textbf{BAz} + \textbf{b}.
\end{split}
\end{align}
Therefore, if $\textbf{AB}\neq\textbf{BA}$, $\tilde{\mathbf{x}}$ does not follow the ICM. The lack of affine equivariance has a far reaching consequence for the ICM, which is referred to as 'outlier propagation'. Outlier propagation means that an outlying cell in a predictor may spread over other components of the corresponding data point, e.g. by linearly combining the predictors in a regression model. From these considerations, we calculate the probability of a row in a $p$-dimensional dataset being contaminated by the formula
\begin{equation}
P_{\mathrm{cont, row}}=1-(1-\epsilon)^{p}.
\label{eq: probability row contamination}
\end{equation}  
So, for any high breakdown point estimator, tuned to have the highest possible breakdown point of 50\%, we obtain the inequality
\begin{align}
\begin{split}
P_{\mathrm{cont, row}}&<0.5\\
\Leftrightarrow 1-(1-\epsilon)^{p}&<0.5\\
\Leftrightarrow\qquad\qquad\,\,\,\,\epsilon&<1-0.5^{p}.
\end{split}
\end{align}
This means that the tolerable fraction of contamination $\epsilon$ in every predictor - for simplicity, we assume that $\epsilon$ is equal for every predictor - is bounded and depends on the dimension of the dataset. From this point of view, the probability of having only  THCM-outliers in the data depends on the dimension of the predictor-matrix and the number of rows as follows:
\begin{align}
&P("\mathrm{THCM}\leadsto \mathrm{ICM}")   \nonumber\\
&= \bigg[1\cdot{n\choose \lceil n\epsilon_{2}\rceil}\cdot{n\choose \lceil n\epsilon_{3}\rceil}\cdot\ldots\cdot{n\choose \lceil n\epsilon_{p}\rceil}\bigg]^{-1}\nonumber \\
&= \bigg[\dfrac{n!}{(n-n\epsilon_{2})!\cdot(n\epsilon_{2})!}\cdot \dfrac{n!}{(n-n\epsilon_{3})!\cdot(n\epsilon_{2})!}\nonumber \\
& \qquad \qquad \ldots\ \cdot \dfrac{n!}{(n-n\epsilon_{2})!\cdot(n\epsilon_{p})!}\bigg]^{-1}
\end{align}
Since the THCM assumes that $\epsilon_{1}=\ldots=\epsilon_{p}=\epsilon$, it follows that
\begin{align}
\label{eq:thcm-icm}
&P("\mathrm{THCM}\leadsto \mathrm{ICM}")   \nonumber\\
&= \bigg[\dfrac{n!}{(n-n\epsilon)!\cdot(n\epsilon)!}\bigg]^{-(p-1)}\nonumber \\
&= {n\choose \lceil n\epsilon\rceil}^{-(p-1)}.
\end{align}
Fig.~\ref{fig: ICM THCM} illustrates how rapidly the probability of having only THCM-outliers decreases for fixed $n$ and $\epsilon$. This, again supports the statement made in \cite{rousseeuw2016detecting} that the outlying  rows  paradigm  is no longer suffcient for modern high-dimensional data sets.
\begin{figure}[htbp]
\Large
\centering
\includegraphics[width=1\linewidth]{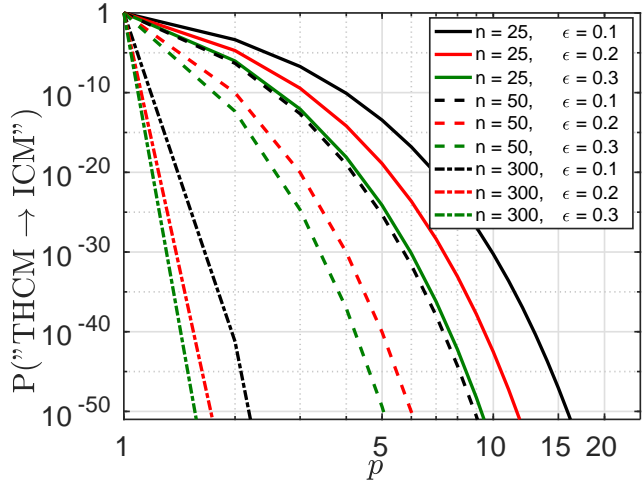}
\caption{Probability of having only THCM-outliers for different values of $n$ and $\epsilon$ as a function of $p$.}
\label{fig: ICM THCM}
\end{figure}


\section{Proposed Methods}
\label{sec:prop-methods}

\subsection{Proposed Definitions of Robust Oracle Properties}
Extending the ideas of \cite{zou2006adaptive} to the ICM, we next introduce what we call the 'robust oracle properties'. We propose a strong and a weak version of the robust oracle properties.
\begin{defn} \text{(Weak Robust Oracle Properties)}
\\
Let $\lbrace \mathbf{x}_{1}, \mathbf{x}_{2}, \ldots, \mathbf{x}_{p} \rbrace$ be the set of predictors,
\begin{equation}
\mathcal{A}\coloneqq\lbrace j:\beta_{j}\neq 0\, \land\,\epsilon_{j}<0.5 \rbrace
\end{equation}
the set of indices corresponding to the set of robust active predictors, and $\mathcal{A}_{n}^{*}\coloneqq\lbrace j:\hat{\beta}_{j}^{*n}\neq 0\rbrace$ the set of indices corresponding to the set of predictors chosen by a Lasso type estimator to be active. Then, a Lasso type estimator needs to have the following properties to be a robust oracle estimator:
\\
\begin{enumerate}
\item Consistency in variable selection: \\$\lim\limits_{n\rightarrow \infty}{P(\mathcal{A}_{n}^{*}=\mathcal{A})=1}$.
\\
\item Asymptotic normality:\\ $\sqrt{n}(\hat{\boldsymbol{\beta}}_{\mathcal{A}}^{*n}-\boldsymbol{\beta}_{\mathcal{A}})\overset{d\,\,}{\rightarrow}\mathcal{N}(\textbf{0},\boldsymbol{\Sigma}^{*})$,
where $\boldsymbol{\Sigma}^{*}$ is the covariance matrix, knowing the true subset model.
\end{enumerate}
\label{def: weak robust oracle properties}
\end{defn}
Loosely speaking, a Lasso type estimator is said to be a robust oracle estimator if it selects only the active variables, while especially leaving out highly outlier-contaminated ones. We call a predictor highly contaminated if at least 50\% of its entries are outliers. The weak robust oracle properties are necessary for any robust estimator, because a single highly contaminated variable might introduce outliers to most observations and lead to a breakdown of the estimator. Note that we differentiate between the {\it active set} and the {\it robust active set} of variables, since active variables should not be chosen, when they are highly contaminated.
\\
\begin{defn} \text{(Strong Robust Oracle Properties)}
\ \\
Let $\lbrace \mathbf{x}_{1}, \mathbf{x}_{2}, \ldots, \mathbf{x}_{p} \rbrace$ be the set of predictors,
\begin{equation}
\mathcal{A}\coloneqq\lbrace j:\beta_{j}\neq 0\, \land\,1-\Pi_{j=1}^{k}(1-(\boldsymbol{\epsilon})_{j:p})<0.5,\, k\in\lbrace 1,\ldots,p\rbrace\rbrace
\end{equation}
the set of indices corresponding to the set of robust active predictors, and $\mathcal{A}_{n}^{*}\coloneqq\lbrace j:\hat{\beta}_{j}^{*n}\neq 0\rbrace$ the set of indices corresponding to the set of predictors chosen by a Lasso type estimator to be active. Then, a Lasso type estimator needs to have the following properties to be a robust oracle estimator:
\\
\begin{enumerate}
\item Consistency in variable selection: \\$\lim\limits_{n\rightarrow \infty}{P(\mathcal{A}_{n}^{*}=\mathcal{A})=1}$.
\\
\item Asymptotic normality:\\ $\sqrt{n}(\hat{\boldsymbol{\beta}}_{\mathcal{A}}^{*n}-\boldsymbol{\beta}_{\mathcal{A}})\overset{d\,\,}{\rightarrow}\mathcal{N}(\textbf{0},\boldsymbol{\Sigma}^{*})$,
where $\boldsymbol{\Sigma}^{*}$ is the covariance matrix, knowing the true subset model.
\end{enumerate}
\label{def: strong robust oracle properties}
\end{defn}
\ \\
Intuitively speaking, the strong robust oracle properties hold for any Lasso type estimator, whose adaptive $\ell_1$-penalty term ensures the penalization of the predictors in an ascending order given by the predictor contamination order statistics
\begin{equation}
(\boldsymbol{\epsilon})_{1:p}\leq (\boldsymbol{\epsilon})_{2:p}\leq\ldots\leq (\boldsymbol{\epsilon})_{p:p},
\label{eq: order contamination}
\end{equation}
while choosing the tuning parameter $\lambda$ in the $\ell_1$-penalty such that predictors enter the model until the breakdown point of the estimator is reached.

\subsection{MM-Robust Weighted Adaptive Lasso (MM-RWAL)}
\label{MM-Robust Weighted Adaptive Lasso (MM-RWAL)}
In this section, we introduce a new method called the MM-Robust Weighted Adaptive Lasso (MM-RWAL). Recently, the MM-Lasso and adaptive MM-Lasso were introduced to robustify against outliers \cite{smucler2015robust}. The objective function of the MM estimator is
\begin{equation}
\hat{\boldsymbol{\beta}}_{\mathrm{MM}}=\underset{\boldsymbol{\beta}}{\argmin}\,\sum\limits_{i=1}^{n}\rho\Bigg(\frac{r_{i}(\boldsymbol{\beta})}{s_{n}(\textbf{r}(\hat{\boldsymbol{\beta}}_{1}))}\Bigg).
\label{eq: MM}
\end{equation}
Here, $\rho(\cdot)$ is a robustifying function (see, e.g. \cite{hampel2005, maronna2006, huber2009}), $\textbf{r}(\hat{\boldsymbol{\beta}}_{1})=\textbf{y}-\mathbf{x}\hat{\boldsymbol{\beta}}_{1}$ is the residual of an S-estimator whose estimates $\hat{\boldsymbol{\beta}}_{1}$ have the property of minimizing a robust M-scale $s_{n}(\textbf{r}(\boldsymbol{\beta}))$ that satisfies
\begin{equation*}
\frac{1}{n}\sum_{i=1}^{n}\rho\left(\frac{r_{i}(\boldsymbol{\beta})}{s_{n}}\right)=b, 
\end{equation*}
where $b$ is usually chosen such that consistency under the Gaussian distribution is obtained. For the MM-(adaptive) Lasso, \eqref{eq: MM} is extended by the penalty terms of \eqref{eq: lasso lagrange} and \eqref{eq: adaptive lasso}, respectively \cite{smucler2015robust}.

The proposed MM-RWAL estimator minimizes an MM objective function to which a robust adaptive $\ell_1$-penalty term is added:
\begin{equation}
\hat{\boldsymbol{\beta}}^{\mathrm{RWAL}}_{\mathrm{MM}}=\underset{\boldsymbol{\beta}}{\argmin}\,\sum\limits_{i=1}^{n}\rho\Bigg(\frac{r_{i}(\boldsymbol{\beta})}{s_{n}(\textbf{r}(\hat{\boldsymbol{\beta}}_{1}))}\Bigg)
+ \lambda_{n}\sum\limits_{j=1}^{p}\hat{w}_{j}|\beta_{j}|.
\label{eq: MM-RWAL}
\end{equation}
Here, the MM-Lasso estimator \cite{smucler2015robust} is used to calculate the weights according to $\hat{w}_{j}=1/|z_{j}\cdot\hat{\beta}_{j \mathrm{MM}}^{\mathrm{Lasso}}|$, with $z_{j}$ defined in \eqref{eq: weights}. 

To robustify the variable selection of the adaptive MM Lasso, we propose to incorporate a measure of outlyingness for each predictor. We use the Stahel Donoho Outlyingness (SDO) \cite{donoho1982breakdown} and adjust it in a similar vein to the Adjusted Stahel Donoho Outlyingness of \cite{van2011stahel}. Let $\textbf{B}=\lbrace \textbf{b}_{1}, \textbf{b}_{2}, \ldots, \textbf{b}_{n} \rbrace\,\subset \mathbb{R}^{p}$ be a set of $n$ observations. Then, the robust Stahel Donoho outlyingness is given by
\begin{equation}
r(\textbf{b}_{i}, \mathbf{x})=\underset{a\in \mathcal{S}_{p}}{\sup}\dfrac{|\textbf{a}^{\top}\textbf{b}_{i} - \mathrm{med}(\textbf{a}^{\top}\mathbf{x})|}{\mathrm{mad}(\textbf{a}^{\top}\mathbf{x})}, \quad i=1,\ldots,n,
\label{eq: SDO}
\end{equation}
where $\mathcal{S}_{p}=\lbrace \textbf{a}\in\mathbb{R}^{p}\, :\, \|\textbf{a}\|_{2}=1 \rbrace$ and $\mathrm{med}(\cdot)$ and $\mathrm{mad}(\cdot)$ denote the median and the median absolute deviation (mad). Since we assume that most rows flagged by the SDO as outliers are not outlying in all of their components, the SDO is extended by also taking into account the outlyingness of the predictors. The idea that has been introduced in \cite{van2011stahel} in a similar vein, is to adjust the SDO of every observation using the outlyingness of every single predictor. This gives us the Predictor Outlyingness (PO)
\begin{equation}
c_{j}=\sum\limits_{i=1}^{n}\dfrac{|x_{ij}-\mathrm{med}(\mathbf{x}_{j})|}{\mathrm{mad}(\mathbf{x}_{j})}, \quad
j=1,\ldots,p.
\label{eq: prediction outlyingness}
\end{equation}
Combining both, the SDO and the PO, we introduce an outlyingness-matrix, whose $(i,j)$-th element is
\begin{equation}
r_{ij}=\alpha r_{i} + (1-\alpha)c_{j},\quad i=1,\ldots,n,\quad j=1,\ldots,p.
\label{eq: outlyingness-matrix}
\end{equation}
We chose the tuning parameter $\alpha$ to be 0.5 throughout this paper, to equally weigh the SDO and the PO, in order to perform well in both contamination models, THCM and ICM.
Finally, by applying a weight function $w(\cdot)$, summing up the rows of the outlyingness matrix and dividing by its cell sum, we obtain the weights
\begin{equation}
z_{j}=\dfrac{p}{\sum\limits_{i=1}^{n}\sum\limits_{j=1}^{p}w(r_{ij})}\sum\limits_{i=1}^{n}w(r_{ij}),\quad j=1,\ldots,p,
\label{eq: weights}
\end{equation}
where $\sum_{j=1}^{p}z_{j}=p$.
\\\\
In order to downweight cells, whose overall outlyingness exceeds a certain threshold, we choose $w(\cdot)$ to be the Huber weight function,
\begin{equation}
w(r)=\mathbbm{1}_{(r\leq c)} + (c/r)^{2}\mathbbm{1}_{(r\leq c)},
\label{eq: huber weight function}
\end{equation}
with $c=\min(\sqrt{\smash[b]{\chi_{p}^{2}(0.5)}},4)$ as proposed in \cite{van2011stahel}.

\subsection{Analysis of the Proposed Predictor Weights}
In this section, we analyze the behavior of the proposed weights $z_{j}$. For this purpose, we define two sets:
\\
\begin{enumerate}
\item[$1.$] $\mathcal{C}_{j}=\lbrace i : w(r_{ij})<1,\, i=1,\ldots,n\rbrace$ 
with cardinality $|\mathcal{C}_{j}|=\lceil n\epsilon_{j}\rceil$, is the set of 
indices corresponding to the contaminated cells in predictor $j$.
\\
\item[$2.$] $\mathcal{E}=\lbrace l : \epsilon_{l}>0,\,l=1,\ldots,p\rbrace$ with 
cardinality $|\mathcal{E}|\eqqcolon \gamma$ is the set of indices corresponding 
to the contaminated predictors, so $\gamma$ is the number of contaminated 
predictors.
\end{enumerate}
\ \\
We will start with the most general formula that describes the behavior of our weights and introduce step by step assumptions, which simplify the equation.
\\\\
Let $\epsilon_{j}$ be the fraction of contamination of the $j$th predictor, and $r_{ij}$ the magnitude of outlyingness of each cell in our designed outlyingness-matrix. Then, the weight of the $j$th predictor is given by
\begin{align}
z_{j}&=\dfrac{p}{\sum\limits_{i=1}^{n}\sum\limits_{j=1}^{p}w(r_{ij})}\sum\limits_{i=1}^{n}w(r_{ij})\\&=\dfrac{p\cdot\bigg[n(1-\epsilon_{j}) + \sum\limits_{i\in \mathcal{C}_{j}}\bigg(\dfrac{c}{r_{ij}}\bigg)^{2}\bigg]}{pn -n\sum\limits_{l\in \mathcal{E}}\epsilon_{l} +\sum\limits_{l\in \mathcal{E}}\sum\limits_{i\in C_{l}}\bigg(\dfrac{c}{r_{il}}\bigg)^{2}}.
\end{align}
Now, we assume that the magnitude of outlyingness is fixed, that is $r_{ij}=r$ for all $i,j$. This yields
\begin{equation}
z_{j}=\dfrac{p\cdot\bigg[1-\epsilon_{j}\bigg(1-\bigg(\dfrac{c}{r}\bigg)^{2}\bigg)\bigg]}{p-\bigg(1-\bigg(\dfrac{c}{r}\bigg)^{2}\bigg)\sum\limits_{l\in \mathcal{E}}\epsilon_{l}}.
\label{eq: same r}
\end{equation}
Eq.~\eqref{eq: same r} provides additional insights that help us understand the behavior of our proposed weights. It shows us: the larger $\epsilon_{j}$ the smaller $z_{j}$. This leads to
\begin{prop}\ \\
Let $\epsilon_{j}\in \mathcal{E}$, $r_{ij}=r$ for all $i,j$ and $\epsilon_{1}\geq\epsilon_{2}\geq\ldots\geq\epsilon_{p}$. Then, it follows that
\begin{equation*}
z_{1}\leq z_{2}\leq\ldots\leq z_{p}.
\end{equation*} 
\end{prop}
In this step, we assume that all contaminated predictors contain the same fraction of contamination, which results in
\begin{equation}
z_{j}=\dfrac{p\cdot\bigg[1-\epsilon\bigg(1-\epsilon\bigg(\dfrac{c}{r}\bigg)^{2}\bigg)\bigg]}{p-\bigg(1-\bigg(\dfrac{c}{r}\bigg)^{2}\bigg)\cdot \gamma\epsilon}.
\label{eq: same epsilon}
\end{equation}
To carry out a plausibility analysis for the derived formulas, we assume now that all $p$-predictors are contaminated $(\gamma=p)$ with the same fraction of outliers $\epsilon$. Intuitively, we expect that all predictors get the same weight and no one preferred over an other one. Since  our weights are designed to sum up to $p$, we expect each predictor $j\in\{1,\ldots,p \} $ to receive the weight $z_{j}=1$. Applying these assumptions to Eq.~\eqref{eq: same epsilon}, we obtain
\begin{equation}
z_{j}=\dfrac{p-p\epsilon\bigg(1-\bigg(\dfrac{c}{r}\bigg)^{2}\bigg)}{p-p\epsilon\bigg(1-\bigg(\dfrac{c}{r}\bigg)^{2}\bigg)}=1,
\end{equation}
which confirms our expectations.
\\\\
The proof of the robust oracle properties in the next section requires the weights $z_{j}$ to be smaller than one. Here, we will prove this property only for Eq. \eqref{eq: same epsilon}.
\begin{prop}\ \\
If $\epsilon_{j}>0$, then $z_{j}\leq 1$ for all $j$.
\label{prop: 1}
\end{prop}
\begin{proof}\ \\
Let $\gamma<p$, then we obtain for fixed $r$ and $\epsilon_{1}=\epsilon_{2}=\ldots=\epsilon_{p}=\epsilon$
\begin{align*}
z_{j}&=\dfrac{p\cdot\bigg[1-\epsilon_{j}\bigg(\dfrac{c}{r}\bigg)^{2}\bigg]}{p-\bigg(1-\bigg(\dfrac{c}{r}\bigg)^{2}\bigg)\cdot \gamma\epsilon}\\
&=\dfrac{p-p\epsilon\bigg(1-\bigg(\dfrac{c}{r}\bigg)^{2}\bigg)}{p-\gamma\epsilon\bigg(1-\bigg(\dfrac{c}{r}\bigg)^{2}\bigg)}\\
&<\dfrac{p-\gamma\epsilon\bigg(1-\bigg(\dfrac{c}{r}\bigg)^{2}\bigg)}{p-\gamma\epsilon\bigg(1-\bigg(\dfrac{c}{r}\bigg)^{2}\bigg)}=1.
\end{align*} 
\end{proof}
\subsection{Proof of the Robust Oracle Properties for the Proposed MM-RWAL}
To prove that the MM-RWAL possesses at least the weak robust oracle properties, we will need some assumptions. Let $G$ denote the distribution of the rows $(x_{i1},x_{i2},\ldots,x_{ip})$ in $\mathbf{X}$ and let $F$ be the distribution of the errors $u_{i}$, which results in the distribution $H$ of the data points $(x_{i1},x_{i2},\ldots,x_{ip},y_{i})$ to become
\begin{equation}
H(\mathbf{x},y)=G(\mathbf{x})F(y-\mathbf{x}^{\top}\boldsymbol{\beta}).
\end{equation}
We now formulate our assumptions as follows:
\begin{enumerate}
\item[\textbf{A1.}]\label{Ass1} All occurring 
$\rho(\cdot)$-functions are twice continuously differentiable and 
$\Psi=\rho^{\prime}$.
\\
\item[\textbf{A2.}]\label{Ass2} The density $f$ of the error terms $u$ is an 
even and monotonically decreasing function of $|u|$.
\\
\item[\textbf{A3.}]\label{Ass3} The second moments of $G$ exist.
\\
\item[\textbf{A4.}]\label{Ass4} The estimator 
$\hat{\boldsymbol{\beta}}_{R}\coloneqq\hat{\boldsymbol{\beta}}^{\mathrm{RWAL}}_{
\mathrm{MM}}$ is a $\sqrt{n}$-consistent estimator of 
$\boldsymbol{\beta}_{R}\coloneqq\boldsymbol{\beta}^{\mathrm{RWAL}}_{\mathrm{MM}}
$, where $\boldsymbol{\beta}_{R}$ complies with Definition \ref{def: weak robust 
oracle properties}.
\\
\item[\textbf{A5.}] $\hat{\boldsymbol{\beta}}_{2}$ is $\sqrt{n}$-consistent.
\end{enumerate}
Note, that we will explicitely prove only the consistency in variable selection here, because the weights $z_{j}$ do not affect the asymptotic normality.
\begin{proof}\textit{(Consistency in Variable Selection)}\\
In this proof, we stick to the notation and structure of Theorem 2 in \cite{huang2008} and Theorem 5 in \cite{smucler2015robust}.
\\\\
With (\textbf{A4.}) and 
$\boldsymbol{\beta}_{0}=(\beta_{0,I},\beta_{0,II})^\top$ being the true 
parameter vector corresponding to Definition \ref{def: weak robust oracle 
properties}, where $I=\mathcal{A}$ are the $s$ indices belonging to the robust 
active predictors and $II=\lbrace 
1,\ldots,p\rbrace\setminus\mathcal{A}=\mathcal{A}^{C}$ is the complementary set 
of $\mathcal{A}$, we know that, with arbitrarily high probability, there exists 
a constant $M_{1}>0$ with
\begin{equation}
\|\hat{\boldsymbol{\beta}}_{R}-\boldsymbol{\beta}_{0}\|<\dfrac{M_{1}}{\sqrt{n}}.
\end{equation}
Now, let 
\begin{align*}
G_{n}(\textbf{u}_{1},\textbf{u}_{2})=\sum\limits_{i=1}^{n}&\rho_{1}\bigg(\dfrac{r_{i}(\boldsymbol{\beta}_{0,I}+\textbf{u}_{1}/\sqrt{n},\, \boldsymbol{\beta}_{0,II}+\textbf{u}_{2}/\sqrt{n})}{s_{n}(\textbf{r}(\hat{\boldsymbol{\beta}}_{1}))}\bigg) \\
& + \lambda_{n}\sum\limits_{j=1}^{s}\dfrac{|\beta_{0,j}+u_{1,j}/\sqrt{n}|}{|\hat{\beta}_{2,j}|}\\
& + \lambda_{n}\sum\limits_{j=s+1}^{p}\dfrac{|u_{2,j-s}/\sqrt{n}|}{|z_{j}\cdot\hat{\beta}_{2,j}|}.
\end{align*}
We obtain $(\hat{\boldsymbol{\beta}}_{R,I}, \hat{\boldsymbol{\beta}}_{R,II})$ by minimizing $G_{n}(\textbf{u}_{1},\textbf{u}_{2})$, subject to $\|\textbf{u}_{1}\| + \|\textbf{u}_{2}\|^{2}\leq M_{1}^{2}$ that we get from:
\begin{align*}
&\|\hat{\boldsymbol{\beta}}_{R} - \boldsymbol{\beta}_{0}\|\leq\dfrac{M_{1}}{\sqrt{n}}\\\\
\Leftrightarrow\quad & \Big\|\boldsymbol{\beta}_{0,I}+\dfrac{\textbf{u}_{1}}{\sqrt{n}}+\boldsymbol{\beta}_{0,II}+\dfrac{\textbf{u}_{2}}{\sqrt{n}}-\boldsymbol{\beta}_{0}\Big\|\leq \dfrac{M_{1}}{\sqrt{n}}\\\\
\Leftrightarrow\quad & \Big\|\dfrac{\textbf{u}_{1} + \textbf{u}_{2}}{\sqrt{n}}\Big\|\leq \dfrac{M_{1}}{\sqrt{n}}\\\\
\Leftrightarrow\quad & \|\textbf{u}_{1}\|^{2} + \|\textbf{u}_{2}\|^{2}\leq M_{1}^{2}.
\end{align*}

We next have to show hat $G_{n}(\textbf{u}_{1},\textbf{u}_{2}) - G_{n}(\textbf{u}_{1},\textbf{0}_{p-s})>0$ holds under the given condition and when $\|\textbf{u}\|>0$. Let 
\begin{equation*}
 G_{n}(\textbf{u}_{1},\textbf{u}_{2}) - G_{n}(\textbf{u}_{1},\textbf{0}_{p-s})=D+E
\end{equation*}
with $D$ and $E$ being defined in Eq.~\eqref{eq: DandE}. 
\begin{figure*}[!t]
 \begin{align}
\label{eq: DandE}
&G_{n}(\textbf{u}_{1},\textbf{u}_{2}) - G_{n}(\textbf{u}_{1},\textbf{0}_{p-s})=\nonumber \\
&=\underbrace{\sum\limits_{i=1}^{n}\bigg[\rho_{1}\bigg(\dfrac{r_{i}(\boldsymbol{\beta}_{0,I} + \textbf{u}_{1}/\sqrt{n},\, \textbf{u}_{2}/\sqrt{n})}{s_{n}(\textbf{r}(\hat{\boldsymbol{\beta}}_{1}))}\bigg) - \rho_{1}\bigg(\dfrac{r_{i}(\boldsymbol{\beta}_{0,I} + \textbf{u}_{1}/\sqrt{n},\, \textbf{0}_{p-s})}{s_{n}(\textbf{r}(\hat{\boldsymbol{\beta}}_{1}))}\bigg)\bigg]}_{\eqqcolon D} + \underbrace{\dfrac{\lambda_{n}}{\sqrt{n}}\sum\limits_{j=s+1}^{p}\dfrac{|u_{2,j-s}|}{|z_{j}\cdot\hat{\beta}_{2,j}|}}_{\eqqcolon E}
 \end{align}
\end{figure*}

With the Mean Value Theorem, we obtain
\begin{equation*}
D=(\textbf{0}_{s},\textbf{u}_{2})^{\top}\dfrac{-1}{\sqrt{n}s_{n}(\textbf{r}(\hat{\boldsymbol{\beta}}_{1}))}\sum\limits_{i=1}^{n}\Psi_{1}\bigg(\dfrac{r_{i}(\boldsymbol{\theta}_{n}^{*})}{s_{n}(\textbf{r}(\hat{\boldsymbol{\beta}}_{1}))}\bigg)\mathbf{x}_{i},
\end{equation*}
where $\boldsymbol{\theta}_{n}^{*}=(\boldsymbol{\beta}_{0,I} + \textbf{u}_{1}/\sqrt{n},\, (1-\alpha_{n})\textbf{u}_{2}/\sqrt{n})$ and $\alpha_{n}\in[0,1]$.
Applying the Mean Value Theorem a second time yields
\begin{align*}
&(\textbf{0}_{s},\textbf{u}_{2})^{\top}\dfrac{-1}{\sqrt{n}s_{n}(\textbf{r}(\hat{\boldsymbol{\beta}}_{1}))}\sum\limits_{i=1}^{n}\Psi_{1}\bigg(\dfrac{r_{i}(\boldsymbol{\theta}_{n}^{*})}{s_{n}(\textbf{r}(\hat{\boldsymbol{\beta}}_{1}))}\bigg)\mathbf{x}_{i}\\\\
&=\dfrac{-1}{\sqrt{n}s_{n}(\textbf{r}(\hat{\boldsymbol{\beta}}_{1}))}(\textbf{0}_{s},\textbf{u}_{2})^{\top}\sum\limits_{i=1}^{n}\Psi_{1}\bigg(\dfrac{r_{i}(\boldsymbol{\beta}_{0})}{s_{n}(\textbf{r}(\hat{\boldsymbol{\beta}}_{1}))}\bigg)\mathbf{x}_{i}\\\\
&\qquad +\dfrac{1}{\sqrt{n}s_{n}^{2}(\textbf{r}(\hat{\boldsymbol{\beta}}_{1}))}(\textbf{0}_{s},\textbf{u}_{2})^{\top}\sum\limits_{i=1}^{n}\Psi_{1}^{\prime}\bigg(\dfrac{r_{i}(\boldsymbol{\theta}_{n}^{**})}{s_{n}(\textbf{r}(\hat{\boldsymbol{\beta}}_{1}))}\bigg)\mathbf{x}_{i}\mathbf{x}_{i}^{\top}\\\\
& \qquad \qquad \qquad \cdot(\textbf{u}_{1}/\sqrt{n},\,(1-\alpha_{n})\textbf{u}_{2}/\sqrt{n})\\\\
&=\dfrac{-1}{\sqrt{n}s_{n}(\textbf{r}(\hat{\boldsymbol{\beta}}_{1}))}(\textbf{0}_{s},\textbf{u}_{2})^{\top}\sum\limits_{i=1}^{n}\Psi_{1}\bigg(\dfrac{r_{i}(\boldsymbol{\beta}_{0})}{s_{n}(\textbf{r}(\hat{\boldsymbol{\beta}}_{1}))}\bigg)\mathbf{x}_{i}\\\\
&\qquad + \dfrac{1}{ns_{n}^{2}(\textbf{r}(\hat{\boldsymbol{\beta}}_{1}))}(\textbf{0}_{s},\textbf{u}_{2})^{\top}\sum\limits_{i=1}^{n}\Psi_{1}^{\prime}\bigg(\dfrac{r_{i}(\boldsymbol{\theta}_{n}^{**})}{s_{n}(\textbf{r}(\hat{\boldsymbol{\beta}}_{1}))}\bigg)\mathbf{x}_{i}\mathbf{x}_{i}^{\top}\\\\
& \qquad \qquad \qquad \cdot (\textbf{u}_{1},(1-\alpha_{n})\textbf{u}_{2})=O_{P}(\|\textbf{u}_{2}\|),
\end{align*}
where $\|\boldsymbol{\theta}_{n}^{**} - 
\boldsymbol{\beta}_{0}\|\leq\|\boldsymbol{\theta}_{n}^{*} - 
\boldsymbol{\beta}_{0}\|$. The last equation follows from the assumptions 
(\textbf{A1.})-(\textbf{A3.}), Lemma 1 in \cite{smucler2015robust} and Lemma~5 
in the Technical Report associated with \cite{yohai1987high}.

Additionally, we have that $E$ is stochastically bounded from below:
\begin{align*}
E&=\dfrac{\lambda_{n}}{\sqrt{n}}\sum\limits_{j=s+1}^{p}\dfrac{|u_{2,j-s}|}{|z_{j}\cdot\hat{\beta}_{2,j}|}=\lambda_{n}\sum\limits_{j=s+1}^{p}\dfrac{|u_{2,j-s}|}{|z_{j}\cdot \sqrt{n}\hat{\beta}_{2,j}|}\\\\
&=\lambda_{n}\cdot\bigg[\sum\limits_{\lbrace j : \beta_{j}=0\land \epsilon_{j}>0\rbrace}\dfrac{|u_{2,j}|}{|z_{j}\cdot \sqrt{n}\hat{\beta}_{2,j}|} \\
&\qquad \qquad \qquad \qquad \qquad+ \sum\limits_{\lbrace j : \beta_{j}=0\land\epsilon_{j}=0\rbrace}\dfrac{|u_{2,j}|}{|z_{j}\cdot \sqrt{n}\hat{\beta}_{2,j}|}\bigg]\\\\
&=\lambda_{n}\cdot\bigg[\sum\limits_{\lbrace j : \beta_{j}=0\land \epsilon_{j}>0\rbrace}\dfrac{|u_{2,j}|}{|z_{j}\cdot \sqrt{n}\hat{\beta}_{2,j}|} \\
&\qquad \qquad \qquad \qquad \qquad + \sum\limits_{\lbrace j : \beta_{j}=0\land\epsilon_{j}=0\rbrace}\dfrac{|u_{2,j}|}{|\sqrt{n}\hat{\beta}_{2,j}|}\bigg]\\\\
\end{align*}
\begin{align*}
&\geq\lambda_{n}\cdot\bigg[\sum\limits_{\lbrace j : \beta_{j}=0\land \epsilon_{j}>0\rbrace}\dfrac{|u_{2,j}|}{|\sqrt{n}\hat{\beta}_{2,j}|}\\\\
&\qquad \qquad \qquad \qquad \qquad + \sum\limits_{\lbrace j : \beta_{j}=0\land\epsilon_{j}=0\rbrace}\dfrac{|u_{2,j}|}{|\sqrt{n}\hat{\beta}_{2,j}|}\bigg]\\\\
&=\lambda_{n}\sum\limits_{j=s+1}^{p}\dfrac{|u_{2,j}|}{|\sqrt{n}\hat{\beta}_{2,j}|}= \lambda_{n}\Omega_{p}(\|\textbf{u}_{2}\|).
\end{align*}
The above inequality follows from Proposition \ref{prop: 1} and the last 
equation follows with assumption (\textbf{A5.}). Now, let $M_{2},M_{3}>0$ be 
some real numbers with $M_{3}\lambda_{n}>M_{2}$, then we have with arbitrarily 
high probability
\begin{align*}
G_{n}(\textbf{u}_{1},\textbf{u}_{2}) - G_{n}(\textbf{u}_{1},\textbf{0}_{p-s})&> -M_{2}\|\textbf{u}_{2}\| + M_{3}\lambda_{n}\|\textbf{u}_{2}\| \\\\
&=\|\textbf{u}_{2}\|\cdot(-M_{2} + M_{3}\lambda_{n})>0
\end{align*}
and the proposition follows for sufficiently large $n$.
\end{proof}

\subsection{Computation of the Weights for MM-RWAL }
The main problem of calculating the weights $z_{j}$ is to compute the supremum in the SDO, because the cardinal number of $\mathcal{S}_{p}$ is infinite and the objective function is non-convex. Therefore, we need to apply a random search algorithm to obtain an approximation of the supremum. We chose to take a subsample from $\mathcal{S}_{p}$ by sampling from a $(p-1)$-dimensional unit-hypershere, as in \cite{peterseny1997uniform}. We use the following algorithm \cite{Machkour2017OCD} to obtain $\mathcal{S}_{p}$ in \eqref{eq: SDO}.
\begin{alg} (Uniform Sampling from a $p$-Dimensional Unit Hypersphere)
\begin{enumerate}
\item Generate $p$ vectors with $k$ entries
\begin{equation}
\mathbf{x}_{j}=(x_{1j},x_{2j},\ldots,x_{kj})^{\top},\quad j=1,\ldots,p,
\end{equation}
where $x_{ij}\sim \mathcal{N}(0,1)$.
\item Calculate $k$ $p$-dimensional vectors
\begin{equation}
\textbf{a}_{i}=\sum\limits_{j=1}^{p}\dfrac{x_{ij}}{\sqrt{x_{i1}^{2} + x_{i2}^{2} + \ldots + x_{ip}^{2}}}\cdot \textbf{e}_{j},\quad i=1,\ldots,k,
\end{equation}
where $\textbf{e}_{j}$ is the $j$th unit vector.
\item Set $\mathcal{S}_{p}\coloneqq\lbrace\textbf{a}_{i}\in\mathbb{R}^{p} : i\in\lbrace 1,\ldots k\rbrace\rbrace$.
\end{enumerate}
\label{alg: unit hypersphere}
\end{alg}

\section{Numerical Experiments}
\label{sec:simulations}
\subsection{Simulation Setup}
Two different Monte Carlo studies are conducted, to assess the performance of the proposed RWAD MM-Lasso.\\
\\
{\noindent \it Scenario 1: $p>n$, correlated predictors, cellwise outliers}\\
A setup with $p=50$ predictors and $n=30$ observations is considered. The regression parameters are defined by $\beta_j=j/5 \;\;\; j\in\{1,\dots,5\}$, while  $\beta_j=0 \;\;\; j\in\{6,\dots, 50\}$. Correlated predictors $\mathbf{x}_j$, $j \in\{1,\dots,p\}$ are generated by sampling from a multivariate zero mean Gaussian distribution with covariance matrix $\Sigma_{ij}=0.5^{|i-j|}$, $\forall i, j \in \{1,\dots ,p\}$. The errors $u_i$ are zero mean i.i.d. Gaussian distributed with variance $\sigma^2=0.5^2$. The responses $y_i$ follow the linear model 
$$y_i=\mathbf{x}_i^{\top}\boldsymbol{\beta} + u_i, \quad i=1,\ldots,n.$$ 
To create cellwise outliers in $\mathbf{X}$,  $\epsilon=\{0\%, 10\%, 20\%, 30\%\}$ of the predictors $\mathbf{x}_j$ are contaminated. For the contaminated predictors,  $30\%$ of the entries are independently and additively contaminated by samples drawn from the distribution $x_{\mathrm{cont}} \sim \mathcal{N}(0,100^2)$.\\
\\
{\noindent \it Scenario 2: $p>n$, correlated predictors, cellwise outliers and additive outliers}\\
The setup is identical to {\it Scenario 1} but additionally, $5\%$ of the responses are additively contaminated by samples drawn from the distribution $y_{\mathrm{cont}} \sim \mathcal{N}(0,100^2)$.

\subsection{Performance Measures}
To assess the performance in terms of parameter estimation and model selection, we display the average mean squared error (MSE)
\begin{equation}
\mathrm{MSE}(\hat{\boldsymbol{\beta}})=\frac{1}{R} \sum_{r=1}^R \frac{1}{p}\sum_{j=1}^p (\beta_j-\hat{\beta}_j^{(r)})^2,
\end{equation}
the average false positive rate (FPR) 
\begin{equation}
\mathrm{FPR}(\hat{\boldsymbol{\beta}})=\frac{1}{R} \sum_{r=1}^R \dfrac{|\lbrace j\in\lbrace 1,\ldots,p\rbrace : \beta_{j}=0\land \hat{\beta}_{j}^{(r)}\neq 0\rbrace|}{|j\in \lbrace 1,\ldots,p\rbrace : \beta_{j}=0|}
\end{equation}
and the average false negative rate (FNR)
\begin{equation}
\mathrm{FNR}(\hat{\boldsymbol{\beta}})=\frac{1}{R} \sum_{r=1}^R \dfrac{|\lbrace j\in\lbrace 1,\ldots,p\rbrace : \beta_{j}\neq 0\land \hat{\beta}_{j}^{(r)}=0\rbrace|}{|j\in \lbrace 1,\ldots,p\rbrace : \beta_{j}\neq 0|},
\end{equation}
where $\hat{\beta}_j^{(r)}$ refers to the parameter estimate of the $r$th Monte Carlo experiment. All results represent averages over $R=100$ Monte Carlo simulations. 
\subsection{Benchmark Methods and Choice of Parameters}
The performance of the MM-RWAL is compared to:
\begin{itemize}
\item OLS Lasso \cite{tibshirani1996regression}
\item OCD Lasso \cite{Machkour2017OCD} 
\item M-Lasso and adaptive M-Lasso \cite{ollila2016}
\item MM-Lasso and adaptive MM-Lasso \cite{smucler2015robust}
\item Sparse LTS \cite{Alfons2013sparseLTS}
\end{itemize}
For all methods, we use a grid of $N_{\lambda}=1000$ candidate regularization parameters, which are equally spaced on the interval $(0,\lambda_{\mathrm{max}}]$. Here, $\lambda_{\mathrm{max}}$ is the value that results in a Lasso estimate for which all regression parameters are equal to zero, i.e., $\hat{\boldsymbol{\beta}}=\mathbf{0}$. $\lambda_{\mathrm{max}}$ is calculated by
\begin{equation}
\lambda_{\mathrm{max}}=\frac{2}{n} \underset{j\in\{1,\dots, p\}}{\max} \mathbf{y}^\top \mathbf{x}_j,
\label{eq:lambda_max}
\end{equation}
using a robust correlation, as described in \cite{Alfons2013sparseLTS}. For all methods, we choose the penalty parameter $\lambda$ that provides the lowest MSE so that the different methods are comparable, independent of the method that is used to select $\lambda$. For the OCD Lasso estimator \cite{Machkour2017OCD} we use a threshold of $t_j=\left(\frac{c_{\mathrm{huber}}}{10 \mathrm{mad}(\mathbf{x}_j)}^2\right)$. For all methods that include the SDO measure, $10^5$ samples are used to approximate the supremum. For the M-Lasso \cite{ollila2016}, Hubers $\rho$-function with a clipping point of $c_{\mathrm{huber}}=1.215$ is used, while for the adaptive M-Lasso we use the bisquare $\rho$-function with a clipping of $c_{\mathrm{bisquare}}=3.44$. The initial estimate is the M-Lasso solution, as proposed in \cite{ollila2016}. For the sparse LTS, we use a subsample proportion of $75\%$. Further, 500 subsamples are used for the first two C-Steps and the best 10 subsample sets are kept to carry out the C-Steps until convergence \cite{Alfons2013sparseLTS}. For the MM-Lasso, the S-Ridge estimator serves as initial estimate, as described in \cite{Maronna2011sridge}. For the MM-Lasso \cite{smucler2015robust}, we use a bisquare $\rho$-function with clipping constant $c_{\mathrm{bisquare}}=3.44$. The adaptive MM-Lasso and MM-RWAL use the same bisquare function and are initialized with the the MM-Lasso estimate.

\subsection{Simulation Results}

Tables~\ref{tab: scenario1} and \ref{tab: scenario2} display the estimation results for {\it Scenarios 1} and {\it Scenarios 2}, respectively. Table~\ref{tab: comp_time} documents the average computation time for one Monte Carlo Run of {\it Scenario 2} for the different methods, using an Intel\textsuperscript{\textregistered} Core\textsuperscript{\texttrademark} 
i7-4510U with 8 GB RAM. In both setups, the OLS Lasso breaks down even for low contamination in either the regressors or the responses and $\mathrm{FNR}(\hat{\boldsymbol{\beta}}_{\mathrm{OLS}})>0.95$since the OLS Lasso selects a regularization parameter for which all coefficients are equal to zero. The MM-RWAL performs best in terms of robust model selection. In all experiments, it maximizes the probability of correctly finding the indices of the non-zero parameters, i.e. $1-(FPR+FNR),$ and in most cases provides the best parameter estimation accuracy in terms of the MSE. A drawback is that it is computationally heavy, which can be contributed mainly to the MM-Lasso algorithm. Future work will investigate combining the RWAL with computationally more efficient estimators.

\begin{table*}[htbp]
\centering
\resizebox{\textwidth}{!}{\begin{tabular}{{lllllllllllll}}
Estimator &  & $\epsilon=0\%$ &  &  & $\epsilon=10\%$ &  &  & $\epsilon=20\%$ &  &  & $\epsilon=30\%$ & \\
\toprule
  & $n\cdot$MSE  & FPR  & FNR & $n\cdot$MSE  & FPR  & FNR& $n\cdot$MSE  & FPR  & FNR& $n\cdot$MSE  & FPR  & FNR \\
  \midrule
  OLS Lasso\cite{tibshirani1996regression} & 0.14&0.11&0.74&1.88&\bf{0.019}&0.95&2.8&\bf{0.019}&0.97&2.76&\bf{0.018}&0.97\\
   \midrule
  OCD Lasso\cite{Machkour2017OCD}  & 0.15&0.11&0.078&1.09&0.034&0.71&1.43&0.06&0.40&0.92&0.09&0.40\\
   \midrule
  M-Lasso \cite{ollila2016}&1.22&0.43&\bf{0.064}&1.25&0.47&\bf{0.11}&1.25&0.48&0.17&1.26&0.48&0.22\\
  \midrule
  ad. M-Lasso \cite{ollila2016}&0.15&0.066&0.12&0.36&0.07&0.24&0.41&0.07&0.30&\bf{0.53}&0.05&0.41\\
  \midrule
  MM-Lasso \cite{smucler2015robust}&0.16&0.19&0.09&0.39&0.27&\bf{0.11}&0.48&0.34&\bf{0.098}&0.74&0.43&\bf{0.16}\\
  \midrule
  ad. MM-Lasso \cite{smucler2015robust}&\bf{0.13}&0.052&0.16&0.34&0.063&0.28&\bf{0.40}&0.073&0.34&0.61&0.084&0.46\\
  \midrule
  \bf{MM-RWAL} &\bf{0.13}&\bf{0.05}&0.09&\bf{0.33}&0.046&\bf{0.11}&\bf{0.40}&0.056&\bf{0.098}&0.62&0.042&\bf{0.16}\\
  \midrule
  sparse LTS \cite{Alfons2013sparseLTS}&0.24&0.12&0.16&0.39&0.11&0.28&0.43&0.10&0.31&0.58&0.092&0.38 \\  
\bottomrule
\end{tabular}}
\vspace{3 pt}
\caption{ $n\cdot$ MSE, FPR and FNR of the estimators for Scenario 1, with $\epsilon$ contaminated predictors. Best results and proposed estimator are highlighted with bold font.}
\label{tab: scenario1}
\end{table*}

\begin{table*}[htbp]
\centering
\resizebox{\textwidth}{!}{\begin{tabular}{{lllllllllllll}}
Estimator &  & $\epsilon=0\%$ &  &  & $\epsilon=10\%$ &  &  & $\epsilon=20\%$ &  &  & $\epsilon=30\%$ & \\
\toprule
  & $n\cdot$MSE  & FPR  & FNR & $n\cdot$MSE  & FPR  & FNR& $n\cdot$MSE  & FPR  & FNR& $n\cdot$MSE  & FPR  & FNR \\
  \midrule
  OLS Lasso \cite{tibshirani1996regression} & 2.24&\bf{0.013}&0.98&1.88&\bf{0.019}&0.95&2.8&0.019&0.97&2.76&\bf{0.018}&0.97\\
   \midrule
  OCD Lasso\cite{Machkour2017OCD} & 1.34&0.016&0.98&1.83&0.022&0.37&2.8&\bf{0.016}&0.97&2.69&0.029&0.95\\
  \midrule
  M-Lasso \cite{ollila2016}& 1.23&0.35&0.16&1.20&0.35&0.16&1.24&0.37&0.21&1.24&0.36&0.26\\
  \midrule
  ad. M-Lasso \cite{ollila2016}&0.46&0.031&0.16&0.52&0.031&0.40&0.59&0.036&0.43&0.75&0.023&0.56\\
  \midrule
  MM-Lasso \cite{smucler2015robust}&0.29&0.17&\bf{0.15}&0.44&0.27&\bf{0.14}&0.59&0.35&\bf{0.15}&0.83&0.46&\bf{0.16}\\
  \midrule
  ad. MM-Lasso \cite{smucler2015robust}&\bf{0.21}&0.038&0.22&\bf{0.29}&0.052&0.28&0.45&0.056&0.38&0.65&0.067&0.48\\
  \midrule
  \bf{MM-RWAL} &\bf{0.21}&0.038&\bf{0.15}&\bf{0.29}&0.042&\bf{0.14}&\bf{0.44}&0.035&\bf{0.15}&0.63&0.036&\bf{0.16}\\
  \midrule
  sparse LTS \cite{Alfons2013sparseLTS}&0.23&0.11&0.17&0.32&0.10&0.22&0.45&0.097&0.32&\bf{0.56}&0.097&0.41\\
\bottomrule
\end{tabular}}
\vspace{3 pt}
\caption{ $n\cdot$MSE, FPR and FNR of the estimators for Scenario 2, with $\epsilon$ contaminated predictors. Best results and proposed estimator are highlighted with bold font.}
\label{tab: scenario2}
\end{table*}

\begin{table}[htbp]
\centering
\begin{tabular}{{ll}}
\toprule
   & ACT [s] \\
   \midrule
  OLS Lasso \cite{tibshirani1996regression}& \bf{0.02}\\
   \midrule
  OCD Lasso \cite{Machkour2017OCD} & 0.97\\
   \midrule
  M-Lasso \cite{ollila2016}& 95.00\\
  \midrule
  ad. M-Lasso \cite{ollila2016}& 97.51\\
   \midrule
  MM-Lasso \cite{smucler2015robust}& 177.63\\
  \midrule
  ad. MM-Lasso \cite{smucler2015robust}&188.97\\
  \midrule
  MM-RWAL & 191.54\\
   \midrule
  sparse LTS \cite{Alfons2013sparseLTS}& 266.79\\
\bottomrule
\end{tabular}
\vspace{3 pt}
\caption{Average Computation Time (ACT) for different estimation methods for one Monte-Carlo run of Scenario 2. Best results and proposed estimator are highlighted with bold font.}
\label{tab: comp_time}
\end{table}

\section{A Real Data Example of Source Estimation for an Atmospheric Inverse Problem}
\label{sec:real-data}
Quantifying the emissions of a pollutant into the atmosphere is essential, for example, in the case of nuclear power plant accidents, volcano eruptions, or to track the releases of greenhouse gases. In this paper, we apply penalized robust estimation to determine the temporal releases of the particles of the European Tracer Experiment (ETEX) at the source location. During the ETEX experiment tracers (perfluorocarbons) were released into the atmosphere in Monterfil, Brittany in 1994. Hourly measurements were taken at 168 ground-level sampling stations in 17 European countries as illustrated in Fig.~\ref{fig: ETEX}. 
\begin{figure}[htbp]
\centering
\Large
\centering
\includegraphics[width=0.95\linewidth]{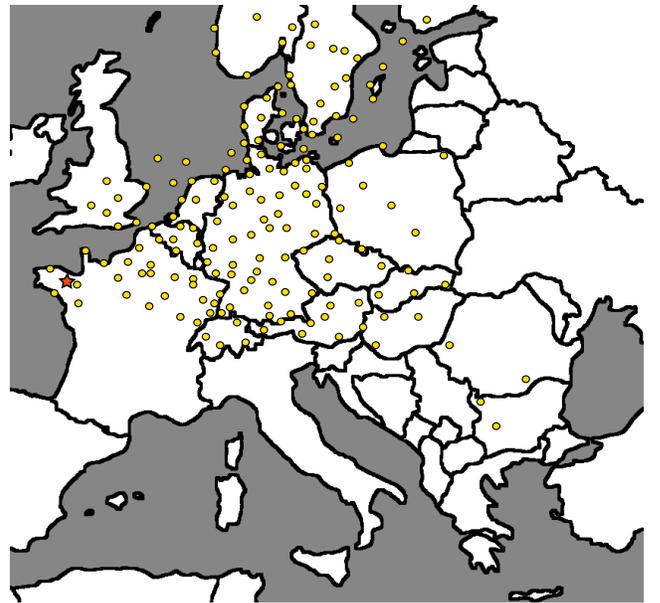}
\caption{Locations of the 168 base stations used in the European Tracer Experiment (ETEX).}
\label{fig: ETEX}
\end{figure}

Atmospheric dispersion models, such as the Lagrangian Particle Dispersion Model (LPDM) allow to formulate the source estimation problem as a linear inverse problem according to Eq.~\eqref{eqn:linear_regression_matrix} as follows. The regression matrix $\mathbf{X}$ is estimated as in \cite{martinez2014robust} by using the Flexible Particle Dispersion Model and is formed by $\mathbf{X}=(\tilde{\mathbf{X}}_1, \dots,\tilde{\mathbf{X}}_k,\ldots, \tilde{\mathbf{X}}_K)^{\top}$, where the $k$th matrix describes the distribution of the particles from the source to the $k$th sensor. Every regression parameter $\beta_j, j\in\{1,\dots,p\}$ corresponds to the amount of PFC that is released by the source at time instant $j$. The sampling interval is one hour and 120 measurements are taken at each sampling station, resulting in $p=120$ unknown regression variables. The responses $\mathbf{y}=(y_1,\dots,y_n)^{\top}=(\tilde{\mathbf{y}}_1^{\top},\dots,\tilde{\mathbf{y}}_k^{\top},\dots,\tilde{\mathbf{y}}_K^{\top})^{\top}$ are a set of stacked observations of the $K=168$ sensors.


The ETEX data is sparse in $\boldsymbol{\beta}$, since only 12 of the regression parameters, i.e. $10\%$, are unequal to zero. The residuals $\mathbf{r}=\mathbf{y}-\mathbf{X}\boldsymbol{\beta}$, given the ground truth values of $\boldsymbol{\beta}$, are non-Gaussian, as displayed by the histogram in Fig.~\ref{fig: hist_y}. Additionally, the regression matrix $\mathbf{X}$ contains outliers, as exemplified for the 68th predictor via the histogram in Fig.~\ref{fig: hist_x}.  Furthermore, $\mathbf{X}$ is sparse, since most of the time, the PFC particles do not reach a sensor, resulting in a regression matrix, which is dominated by zero-valued cells. 
\begin{figure}
\centering
\begin{minipage}{0.4\textwidth}
\includegraphics[width=\textwidth]{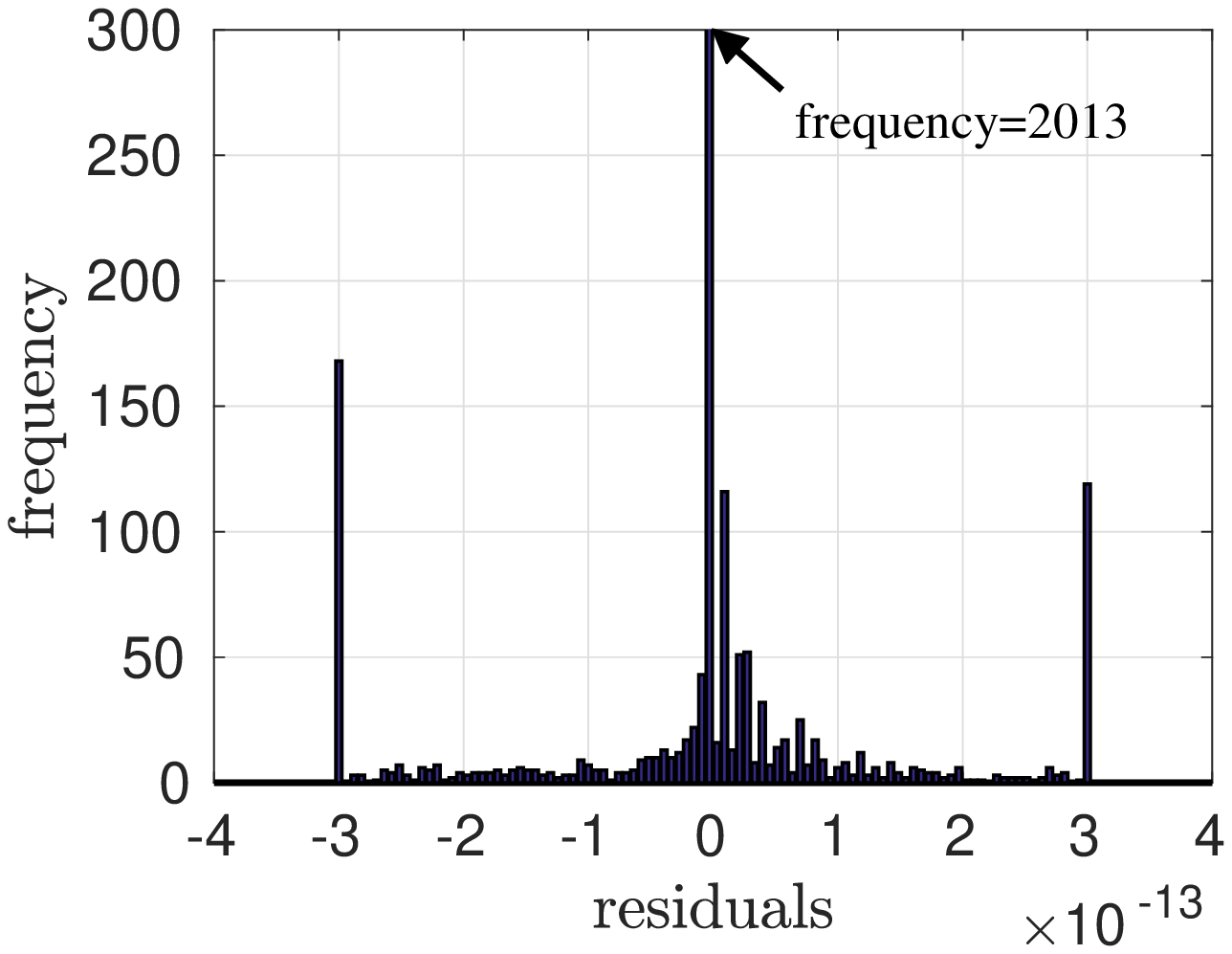}
\caption{Histogram of the residuals.}
\label{fig: hist_y}
\end{minipage}
\hfill
\begin{minipage}{0.4\textwidth}
\includegraphics[width=\textwidth]{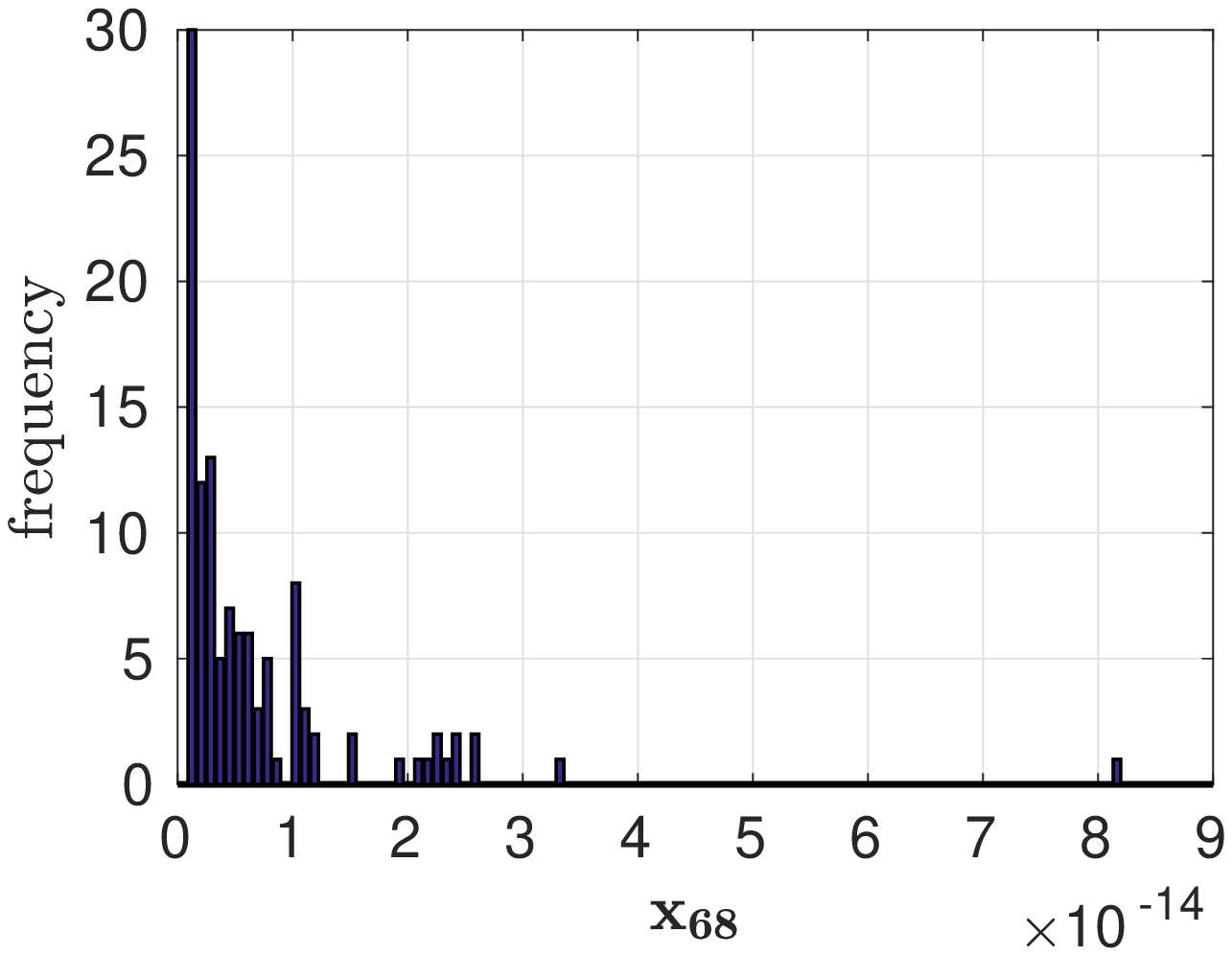}
\caption{Histogram of the 68th predictor $\mathbf{x}_{68}$ (zero bin omitted).}
\label{fig: hist_x}
\end{minipage}
\end{figure}

\subsection{Pre-Processing of the ETEX data}
The following preprocessing steps are applied to the data:
\begin{enumerate}
\item Remove data points $(\mathbf{x}_j,y_j)^\top$, where all entries of the predictor $\mathbf{x}_j$ are equal to zero.
\item Normalize the data robustly using the mad and the median,
$$\mathbf{y}\leftarrow \frac{\mathbf{y}-\mathrm{med}(\mathbf{y})}{\mathrm{mad}(\mathbf{y})}$$
$$\mathbf{x}_j\leftarrow \frac{\mathbf{x}_j-\mathrm{med}(\mathbf{x}_j)}{\mathrm{mad}(\mathbf{x}_j)}.$$
\item When applying the median or the mad to the predictors $\mathbf{x}_j$, only use samples which are greater than zero\footnote{The predictors contain an overwhelming number of components which are zero since $\mathbf{X}$ is highly sparse. Robust estimates like the median or the mad will result in a value of zero if more than $50\%$ of the entries of $\mathbf{x}_j$ are zero. Obviously, taking only the positive components into account leads to estimates that are based on the distribution of the data that we are actually interested in.}. 

\item Apply a robust PCA \cite{Croux2007PCA} and reconstruct $\mathbf{X}$ using only $N_p$ principal components, such that the mad of the $N_p$ principal components corresponds to $90\%$ of the total mad
\footnote{The reason for applying a robust PCA is to provide for all algorithms a matrix that is not as badly conditioned as the original regression matrix. The value $90\%$ has been empirically determined from a range of possible values between $85\%-99\%$.}.
\item Further, since we know that the number of particles omitted by the source can only be positive $\beta_j\geq 0$, we impose a non-negativity constraint on the parameters $\beta_j$, $j \in\{1,\dots,p\}$. 
\end{enumerate}
This leads to the positive Lasso.
\subsection{The Positive MM-RWAL Estimator}
\begin{defn} (The Positive Lasso Estimator)
\begin{align}
\hat{\boldsymbol{\beta}}_{\mathrm{pos\, Lasso}}=\underset{\boldsymbol{\beta}}{\argmin} ||\mathbf{y}-\mathbf{X}\boldsymbol{\beta}||_2^2,\;\;\nonumber \\
\text{subject to }\; ||\boldsymbol{\beta}||_1\leq t\; \mathrm{ and } \; \beta_j\geq 0 \;\;\; j\in\{1,\dots,p\}.
\end{align}
\end{defn}
The positive Lasso can be calculated using a modified cyclic coordinate descent algorithm as follows.
\begin{alg} (The Positive Lasso Estimator Using Cyclic Coordinate Descent)\\
\begin{enumerate}
\item Standardize the regressors so that $\sum_i x_{ij}/n=0$ and $\sum_i x_{ij}^2=1$.
\item Initialize the regression parameters with an arbitrary value, e.g.,
$$\beta_j=0\;\; j \in \{1,\dots,p\}.$$
\item Calculate the update of the $j$th regression parameter by keeping the parameters $k\neq j$ fixed. Start with $j=1$ and end with $j=p$. If a regression parameter is negative replace it by zero
$$\hat{\beta}_j=\max\left\lbrace \mathrm{\mathsf{S}} \left(\sum_{i=1}^n x_{ij}(y_i-\sum_{k\neq j} x_{ik} \hat{\beta_k} ),\lambda \right), 0\right\rbrace.$$
\item Repeat step 3 until convergence.
\end{enumerate}
\label{alg:postive_lasso_ccd}
\end{alg}
Here, $\mathsf{S}(x, \lambda)=\mathrm{sign}(x)(|x|-\lambda)_+$ is the soft thresholding function with $(\cdot)_+=\mathrm{max} \{\cdot,0\}$.

The LARS algorithm can be modified in a similar way, as proposed in \cite{Efron2004LARS}. 
Therewith, Algorithm \ref{alg:postive_lasso_ccd} is modified to compute the positive MM-Lasso estimates,
\begin{align*}
 \hat{\boldsymbol{\beta}}_{\mathrm{MM}}=\underset{\boldsymbol{\beta}}{\argmin}\sum_{i=1}^n \rho\left(\frac{y_i-\mathbf{x}_i\boldsymbol{\beta}}{\hat{\sigma}}\right)+\lambda \sum_{j=1}^p |\beta_j|,\;\;\\
\text{subject to }\; \beta_j\geq 0 \;\;\; j\in\{1,\dots,p\}
\end{align*}
and the extension to the positive MM-RWAL is straightforward.

\begin{figure*}[!ht]
\centering
\includegraphics[width=.35\textwidth]{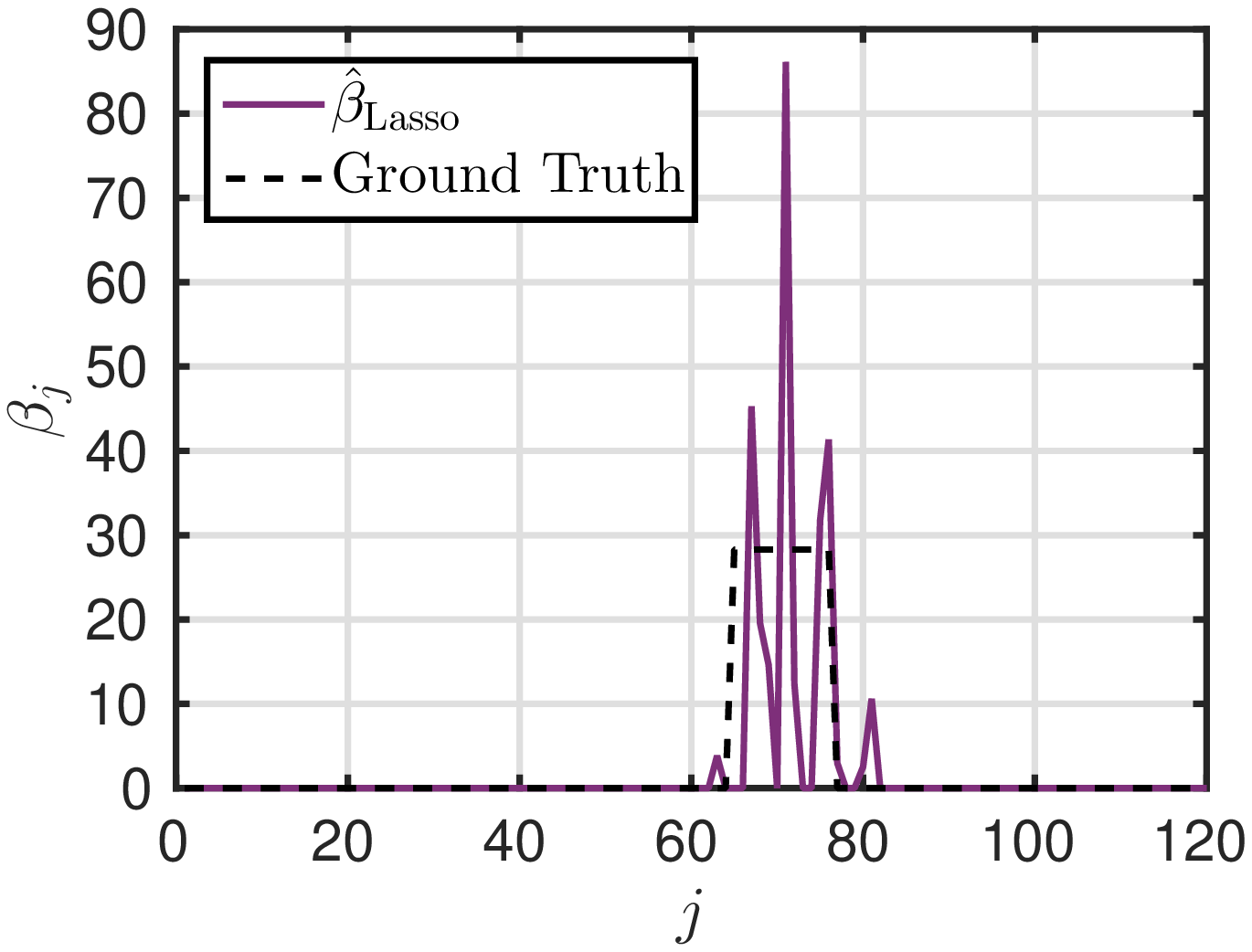}
\includegraphics[width=.35\textwidth]{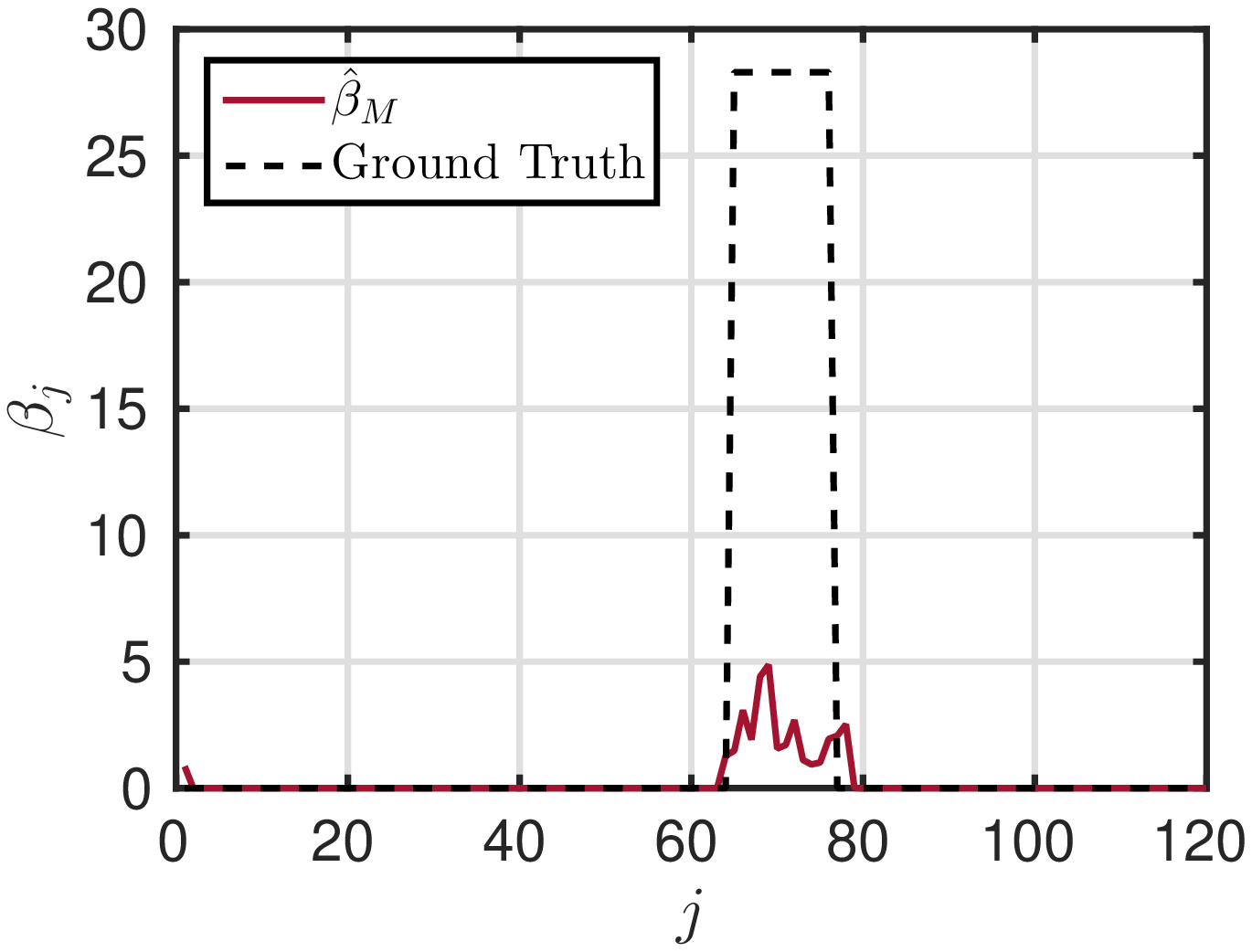}
\includegraphics[width=.35\textwidth]{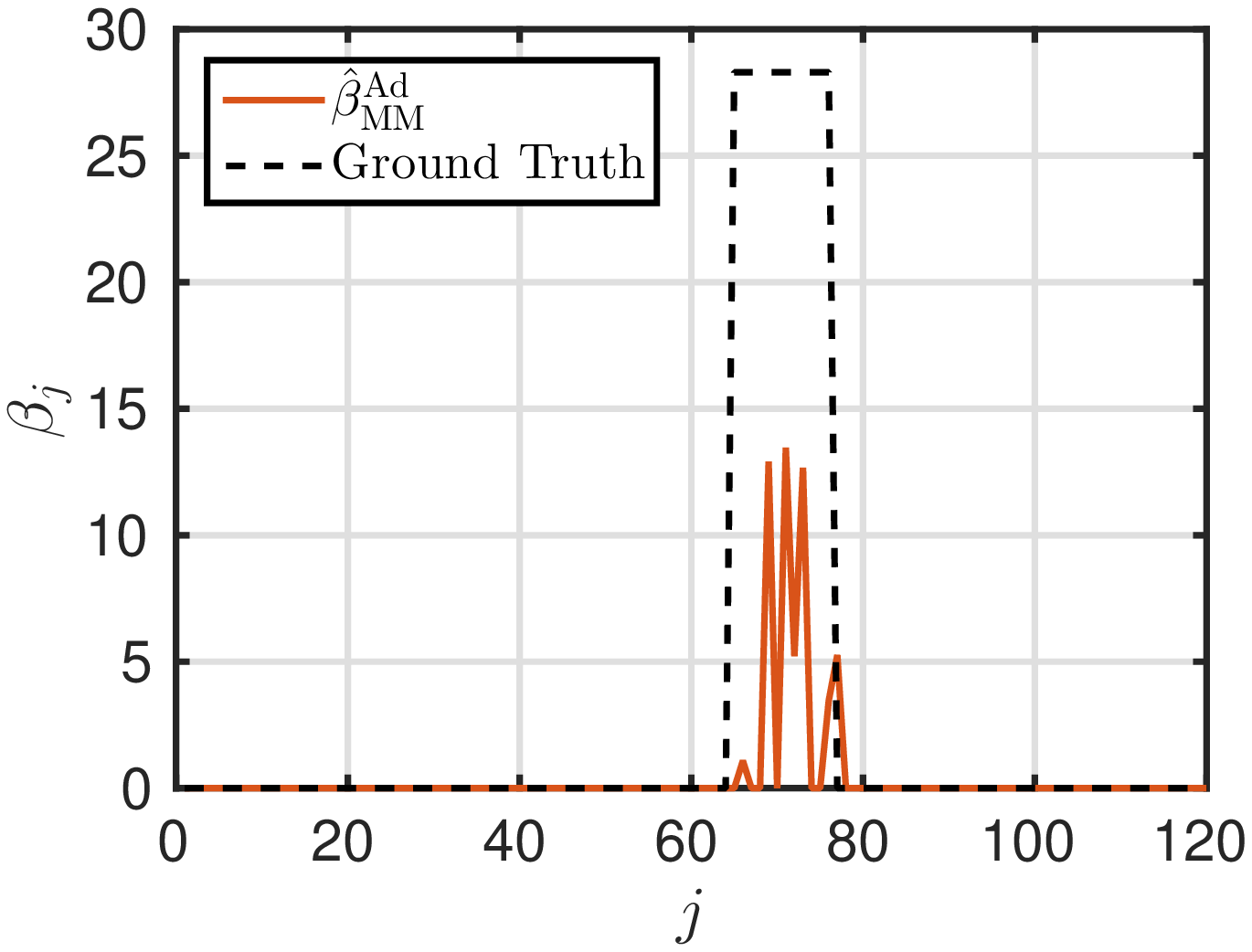}
\includegraphics[width=.35\textwidth]{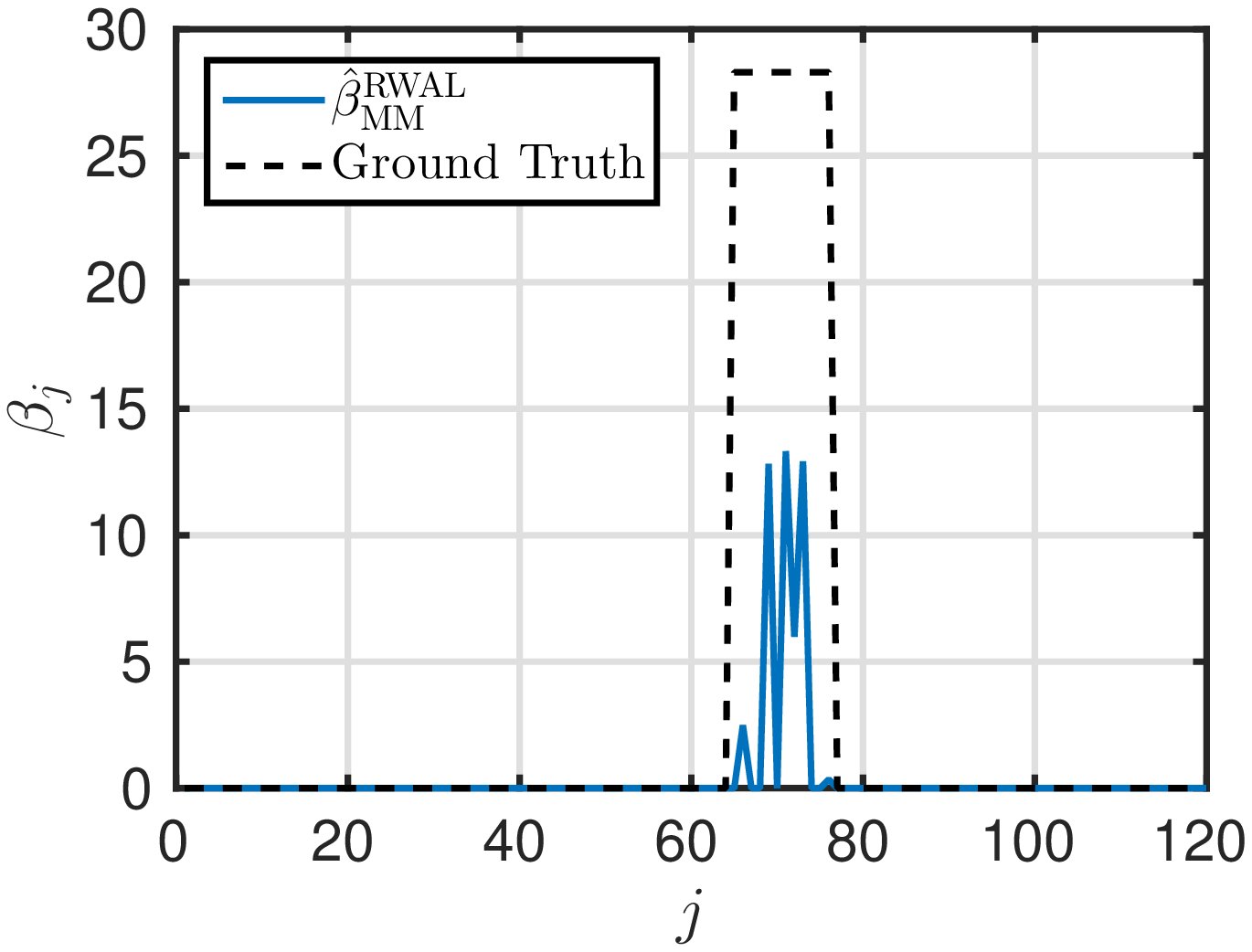}
\caption{Estimated source emissions for the Lasso, Ad. M-Lasso, Ad. MM-Lasso and MM-RWAL.}
\label{fig: MM_real_data}
\end{figure*}

\subsection{Parameter Selection and Performance Metrics}
All parameters are set as described in Section \ref{sec:simulations} and the optimal regularization parameter $\lambda^*$ is found by using the parameter with the lowest MSE compared to the ground truth:
$$\lambda^*=\underset{\lambda \in \Lambda}{\argmin} \frac{1}{p}\sum_{j=1}^p (\beta_j-\hat{\beta}(\lambda))^2$$ 
The performance metrics are the mean squared error (MSE)
\begin{equation}
\mathrm{MSE}(\hat{\boldsymbol{\beta}})=\frac{1}{p}\sum_{j=1}^p (\beta_j-\hat{\beta}_j)^2,
\end{equation}
the false positive rate (FPR)
\begin{equation}
\mathrm{FPR}(\hat{\boldsymbol{\beta}})=\dfrac{|\lbrace j\in\lbrace 1,\ldots,p\rbrace : \beta_{j}=0\land \hat{\beta}_{j}\neq 0\rbrace|}{|j\in \lbrace 1,\ldots,p\rbrace : \beta_{j}=0|}
\label{eq: FPR}
\end{equation}
and the false negative rate (FNR)
\begin{equation}
\mathrm{FNR}(\hat{\boldsymbol{\beta}})=\dfrac{|\lbrace j\in\lbrace 1,\ldots,p\rbrace : \beta_{j}\neq 0\land \hat{\beta}_{j}=0\rbrace|}{|j\in \lbrace 1,\ldots,p\rbrace : \beta_{j}\neq 0|}.
\label{eq: FNR}
\end{equation}
The false positive rate measures how many time instances are flagged as containing a source emission, while there was actually none. The false negative rate measures how many time instances containing a source emission are falsely left out.

\begin{table}
\centering
\begin{tabular}{{lllll}}
  & MSE  & FPR  & FNR  & 1-(FPR+FNR)\\
\toprule
OLS Lasso \cite{tibshirani1996regression} &  70.66 &  0.046 &  0.36 & 0.594\\
\midrule
OCD Lasso \cite{Machkour2017OCD} & $3.75 \cdot 10^4$    & {\bf 0.029} &  0.438 & 0.533 \\
\midrule
Ad. M-Lasso  \cite{ollila2016} & 68.23 &  \bf{0} & 0.25 & 0.75\\
\midrule
MM-Lasso \cite{smucler2015robust}& {\bf 59.84} &  0.039 & 0.5& 0.461 \\
\midrule
Ad. MM-Lasso \cite{smucler2015robust} & 61.88 &  0.053 & 0.143 & 0.804\\
\midrule
{\bf MM-RWAL}  & 62.1 &  0.053 & \bf{0}& \bf{0.947} \\
\bottomrule
\end{tabular}
\vspace{3 pt}
\caption{MSE, FPR and FNR of the ETEX data using different estimation methods. Best results and proposed estimator are highlighted with bold font.}
\label{tab: real_data}
\end{table}
\subsection{Results for the ETEX Experimental Data}
Table~\ref{tab: real_data} displays the MSE, FPR and FNR for the ETEX data set, while Fig.~\ref{fig: MM_real_data} shows the estimated source emissions for the Lasso, Ad. M-Lasso, Ad. MM-Lasso and MM-RWAL. In general, the difference in the MSE is not very large between all estimators, except for the OCD-Lasso. For this example, even the OLS Lasso estimate results in a reasonable estimation, mainly thanks to robust PCA.  However, the OLS Lasso estimate still has poor model selection properties. The OCD Lasso estimates are dominated by outlying predictors rendering this estimator useless for the above real-data application. The MM-RWAL is the only estimator, which correctly detects all coefficients that are equal to zero, i.e. FNR=0. The MM-RWAL also by far outperforms its competitors in correctly finding the indices of the non-zero parameters, i.e. $1-(\mathrm{FPR}+\mathrm{FNR})=0.947.$

\section{Conclusion}
\label{sec:conclusion}
The problem of finding sparse solutions to under-determined, or ill-conditioned, linear regression problems that are contaminated by cellwise and rowwise outliers was investigated. We defined 'robust oracle properties' that are required to perform robust variable selection for such models. We introduced and analyzed a robustly weighted and adaptive Lasso type regularization term and integrated it into the objective function of the MM-estimator, resulting in the proposed \textit{MM-Robust Weighted Adaptive Lasso (MM-RWAL)} for which we showed that at least the weak robust oracle properties hold. An algorithm to compute the weights was proposed and analyzed. The MM-RWAL outperformed existing robust sparse estimators in numerical experiments and proved its usefulness in a real-data application of estimating the sparse non-negative spatio-temporal emissions of a pollutant, given noisy observations and an imprecisely estimated ill-conditioned and sparse dispersion model containing cellwise and rowwise outliers. In future work, the proposed RWAL penalty can easily be integrated into the objective function of other rowwise robust estimators to extend them to the cellwise contamination framework.

\section*{Acknowledgment}
The authors would like to thank Marta Martinez-Camara and Martin Vetterli for making us aware of the ETEX experiment and for many interesting discussions on this application.

\bibliographystyle{IEEEbib}

%
%
%
%
%
%
%
%

\end{document}